\documentclass{tac}

\usepackage{amssymb,amsmath,stmaryrd,mathrsfs}
\usepackage[all,2cell]{xy}
\UseAllTwocells
\usepackage{enumitem}
\usepackage{mathtools}          
\usepackage{microtype}
\usepackage{url}                
\usepackage{xspace}             

\usepackage[colorlinks=true]{hyperref}
\hypersetup{allcolors=[rgb]{0.1,0.1,0.4}}

\author{Michael Shulman}
\thanks{This material is based on research sponsored by The United
  States Air Force Research Laboratory under agreement number
  FA9550-15-1-0053.  The U.S.~Government is authorized to reproduce
  and distribute reprints for Governmental purposes notwithstanding
  any copyright notation thereon.  The views and conclusions contained
  herein are those of the author and should not be interpreted as
  necessarily representing the official policies or endorsements,
  either expressed or implied, of the United States Air Force Research
  Laboratory, the U.S.~Government, or Carnegie Mellon University.}
\title{Contravariance through enrichment}
\copyrightyear{2018}
\eaddress{shulman@sandiego.edu}

\makeatletter
\let\ea\expandafter

\newcount\foreachcount

\def\foreachletter#1#2#3{\foreachcount=#1
  \ea\loop\ea\ea\ea#3\@alph\foreachcount
  \advance\foreachcount by 1
  \ifnum\foreachcount<#2\repeat}

\def\foreachLetter#1#2#3{\foreachcount=#1
  \ea\loop\ea\ea\ea#3\@Alph\foreachcount
  \advance\foreachcount by 1
  \ifnum\foreachcount<#2\repeat}

\def\definecal#1{\ea\gdef\csname c#1\endcsname{\ensuremath{\mathcal{#1}}\xspace}}
\foreachLetter{1}{27}{\definecal}
\def\definebold#1{\ea\gdef\csname b#1\endcsname{\ensuremath{\mathbf{#1}}\xspace}}
\foreachLetter{1}{27}{\definebold}
\def\definebb#1{\ea\gdef\csname d#1\endcsname{\ensuremath{\mathbb{#1}}\xspace}}
\foreachLetter{1}{27}{\definebb}
\def\definefrak#1{\ea\gdef\csname f#1\endcsname{\ensuremath{\mathfrak{#1}}\xspace}}
\foreachletter{1}{9}{\definefrak}
\foreachletter{10}{27}{\definefrak}
\foreachLetter{1}{27}{\definefrak}
\def\definebar#1{\ea\gdef\csname #1bar\endcsname{\ensuremath{\overline{#1}}\xspace}}
\foreachLetter{1}{27}{\definebar}
\foreachletter{1}{8}{\definebar} 
\foreachletter{9}{15}{\definebar} 
\foreachletter{16}{27}{\definebar}
\def\defineul#1{\ea\gdef\csname u#1\endcsname{\ensuremath{\underline{#1}}\xspace}}
\foreachLetter{1}{27}{\defineul}
\foreachletter{1}{27}{\defineul}

\def\autofmt@b#1\autofmt@end{\mathbf{#1}}
\def\autofmt@d#1#2\autofmt@end{\mathbb{#1}\mathsf{#2}}
\def\autofmt@c#1#2\autofmt@end{\mathcal{#1}\mathit{#2}}
\def\autofmt@u#1\autofmt@end{\underline{\smash{\mathsf{#1}}}}

\def\auto@drop#1{}
\def\autodef#1{\ea\ea\ea\@autodef\ea\ea\ea#1\ea\auto@drop\string#1\autodef@end}
\def\@autodef#1#2#3\autodef@end{%
  \ea\def\ea#1\ea{\ea\ensuremath\ea{\csname autofmt@#2\endcsname#3\autofmt@end}\xspace}}
\def\autodefs@end{blarg!}
\def\autodefs#1{\@autodefs#1\autodefs@end}
\def\@autodefs#1{\ifx#1\autodefs@end%
  \def\autodefs@next{}%
  \else%
  \def\autodefs@next{\autodef#1\@autodefs}%
  \fi\autodefs@next}

\autodefs{\cCat\cMonCat\bCat\bGpd\cDInv\uCat\cMap\cSpan\cProf\cSpan\bSet\bDInv\cMat\cMod\dProf\dMat\dMod\cAlg\cPs}

\DeclareSymbolFont{bbold}{U}{bbold}{m}{n}
\DeclareSymbolFontAlphabet{\mathbbb}{bbold}
\newcommand{\done}{\ensuremath{\mathbbb{1}}\xspace}

\newcommand{\op}{^{\mathrm{op}}}
\newcommand{\co}{^{\mathrm{co}}}
\newcommand{\coop}{^{\mathrm{coop}}}
\let\ten\otimes

\newcommand{\too}[1][]{\ensuremath{\overset{#1}{\longrightarrow}}}
\let\xto\xrightarrow
\def\toiso{\xto{\smash{\raisebox{-.5mm}{$\scriptstyle\sim$}}}}

\def\slashedarrowfill@#1#2#3#4#5{%
  $\m@th\thickmuskip0mu\medmuskip\thickmuskip\thinmuskip\thickmuskip
   \relax#5#1\mkern-7mu%
   \cleaders\hbox{$#5\mkern-2mu#2\mkern-2mu$}\hfill
   \mathclap{#3}\mathclap{#2}%
   \cleaders\hbox{$#5\mkern-2mu#2\mkern-2mu$}\hfill
   \mkern-7mu#4$%
}
\def\rightslashedarrowfill@{%
  \slashedarrowfill@\relbar\relbar\mapstochar\rightarrow}
\newcommand\xslashedrightarrow[2][]{%
  \ext@arrow 0055{\rightslashedarrowfill@}{#1}{#2}}
\newcommand{\hto}{\xslashedrightarrow{}}
\let\xhto\xslashedrightarrow

  \let\your@state\state
  \def\state#1{\my@state#1}
  \def\my@state#1.{\gdef\currthmtype{#1}\your@state{#1.}}
  \let\your@staterm\staterm
  \def\staterm#1{\my@staterm#1}
  \def\my@staterm#1.{\gdef\currthmtype{#1}\your@staterm{#1.}}
  \let\@defthm\newtheorem
  \def\switchtotheoremrm{\let\@defthm\newtheoremrm}
  \def\defthm#1#2#3{\@defthm{#1}{#2}} 
  \def\currthmtype{}
  \def\autoref#1{\ref*{label@name@#1}~\ref{#1}}
  \let\your@section\section
  \def\section{\gdef\currthmtype{section}\your@section}
  \AtBeginDocument{%
    \let\old@label\label%
    \def\label#1{%
      {\let\your@currentlabel\@currentlabel%
        \edef\@currentlabel{\currthmtype}%
        \old@label{label@name@#1}}%
      \old@label{#1}}
  }
  \let\cref\autoref

\newtheorem{thm}{Theorem}[section]

\defthm{cor}{Corollary}{Corollaries}
\defthm{prop}{Proposition}{Propositions}
\defthm{lem}{Lemma}{Lemmas}
\defthm{sch}{Scholium}{Scholia}
\defthm{defn}{Definition}{Definitions}
\switchtotheoremrm
\defthm{rmk}{Remark}{Remarks}
\defthm{eg}{Example}{Examples}

\let\qed\endproof

\setitemize[1]{leftmargin=2em}
\setenumerate[1]{leftmargin=*}

\let\c@equation\c@subsection
\numberwithin{equation}{section}

\@ifpackageloaded{mathtools}{\mathtoolsset{showonlyrefs,showmanualtags}}{}

\let\al\alpha
\let\be\beta
\let\gm\gamma
\let\ep\varepsilon

\let\om\omega

\let\ze\zeta
\let\th\theta

\makeatother

\let\V\bV
\let\A\uA
\let\B\uB

\newcommand{\fii}{\mathfrak{i}}
\renewcommand{\o}{^\circ}
\newcommand{\oprime}{^\circ{}'}
\newcommand{\ow}{^{\lozenge}}
\newcommand{\oo}{\o{}\o}
\newcommand{\ooo}{\o{}\o{}\o}
\newcommand{\oooo}{\o{}\o{}\o{}\o}

\newcommand{\p}{^+}
\newcommand{\m}{^-}
\newcommand{\unit}{\done}
\newcommand{\du}{\done^{-}}
\newcommand{\twu}[1]{\done^{#1}}
\newcommand{\ob}{\operatorname{ob}}
\newcommand{\homr}[2]{#1\mathrel{\fatbslash} #2}
\newcommand{\homl}[2]{#2 \mathrel{\!\!\fatslash} #1}
\newcommand{\homrbare}{\fatbslash}
\newcommand{\homlbare}{\!\!\fatslash\,}
\newcommand{\pow}[2]{#2 \mathrel{\oslash} #1}

\newcommand{\wcat}{\ensuremath{\mathbf{W}\text{-}\cCat}\xspace}
\newcommand{\wprof}{\ensuremath{\mathbf{W}\text{-}\dProf}\xspace}
\newcommand{\wmat}{\ensuremath{\mathbf{W}\text{-}\dMat}\xspace}
\newcommand{\hkl}{\ensuremath{\mathbb{H}\text{-}\mathsf{Kl}}}
\newcommand{\twgw}{\ensuremath{\textstyle\int_G\!\bW}\xspace}
\newcommand{\twgcw}{\ensuremath{\textstyle\int_G\!\cW}\xspace}
\newcommand{\twgwmat}{\ensuremath{(\twgw)\text{-}\dMat}\xspace}
\newcommand{\twocat}{2\text-\cCat}
\newcommand{\act}[2]{{#2}^{#1}}
\newcommand{\pact}[2]{{(#2)}^{#1}}
\newcommand{\Bigpact}[2]{\Big(#2\Big)^{#1}}
\newcommand{\awinv}{\ensuremath{\cA[\cW^{-1}]}\xspace}
\newcommand{\tovar}[2][]{\xrightarrow[#2]{#1}}
\newcommand{\alg}{\text{-}\cAlg}
\newcommand{\algs}{\text{-}\cAlg_s}
\newcommand{\psalg}{\text{-}\cPs\cAlg}

\begin{document}
\maketitle

\begin{abstract}
  We define strict and weak duality involutions on 2-categories, and
  prove a coherence theorem that every bicategory with a weak duality
  involution is biequivalent to a 2-category with a strict duality
  involution.  For this purpose we introduce ``2-categories with
  contravariance'', a sort of enhanced 2-category with a basic notion
  of ``contravariant morphism'', which can be regarded either as
  generalized multicategories or as enriched categories.  This enables
  a universal characterization of duality involutions using absolute
  weighted colimits, leading to a conceptual proof of the coherence
  theorem.
\end{abstract}

\tableofcontents

\section{Introduction}
\label{sec:introduction}

One of the more mysterious bits of structure possessed by the
2-category \cCat is its \emph{duality involution}
\[ (-)\op : \cCat\co \to \cCat. \]
(As usual, the notation $(-)\co$ denotes reversal of 2-cells but not
1-cells.)  Many familiar 2-categories possess similar involutions,
such as 2-categories of enriched or internal categories, the
2-category of monoidal categories and strong monoidal functors, or
$[\A,\cCat]$ whenever \A is a locally groupoidal 2-category; and they
are an essential part of much standard category theory.

However, there does not yet exist a complete abstract theory of such
``duality involutions''.  A big step forward was the observation by
Day and Street~\cite{ds:monbi-hopfagbd} that $A\op$ is a monoidal dual
of $A$ in the monoidal bicategory of profunctors.  As important and
useful as this fact is, it does not exhaust the properties of
$(-)\op$; indeed, it does not even determine $A\op$ up to equivalence!

In this paper we study duality involutions like $(-)\op$ acting on
2-categories like \cCat, rather than bicategories like \cProf.  (We
leave it for future work to combine the two, perhaps with a theory of
``duality involutions on proarrow equipments''.  One step in that
direction was taken by~\cite{weber:2toposes}, in the case where
profunctors are represented by discrete two-sided fibrations.)  Note
that in most of the examples cited above, $(-)\op$ is a 2-functor that
is a \emph{strict} involution, in that we have $(A\op)\op = A$ on the
nose.  On the other hand, from a higher-categorical perspective it
would be more natural to ask only for a \emph{weak} duality
involution, where $(-)\op$ is a pseudofunctor that is self-inverse up
to coherent pseudo-natural equivalence.  For instance, strict duality
involutions are not preserved by passage to a biequivalent bicategory,
but weak ones are.

The main result of this paper is that there is no loss of generality
in considering only strict involutions.  More precisely, we prove the
following coherence theorem.

\begin{thm}\label{thm:intro} Every bicategory with a weak duality
involution is biequivalent to a 2-category with a strict duality
involution, by a biequivalence which respects the involutions up to
coherent equivalence.
\end{thm}

Let me now say a few words about the proof of \autoref{thm:intro},
which I regard as more interesting than its statement.  Often, when
proving a coherence theorem for categorical structure at the level of
\emph{objects}, it is helpful to consider first an additional
structure at the level of \emph{morphisms}, whose presence enables the
object-level structure to be characterized by a universal property.
For instance, instead of pseudofunctors $A\op\to \cCat$, we may
consider categories over $A$, among which those underlying some
pseudofunctor (the fibrations) are characterized by the existence of
cartesian arrows, which have a universal property.  Similarly, instead
of monoidal categories, we may consider multicategories, among which
those underlying some monoidal category are characterized by the
existence of representing objects, which also have a universal
property.

An abstract framework for this procedure is the theory of
\emph{generalized multicategories};
see~\cite{hermida:coh-univ,cs:multicats} and the numerous other
references in~\cite{cs:multicats}.  In general, for a suitably nice
2-monad $T$, in addition to the usual notions of strict and pseudo
$T$-algebra, there is a notion of \emph{virtual} $T$-algebra, which
contains additional kinds of morphisms whose domain ``ought to be an
object given by a $T$-action if such existed''.  For example, if $T$
is the 2-monad for strict monoidal categories, then a virtual
$T$-algebra is an ordinary multicategory, in which there are
``multimorphisms'' whose domains are finite lists of objects that
``ought to be tensor products if we had a monoidal category''.

In our case, it is easy to write down a 2-monad whose strict algebras
are 2-categories with a strict duality involution: it is
$T\cA = \cA + \cA\co$.  A virtual algebra for this 2-monad is, roughly
speaking, a 2-category equipped with a basic notion of ``contravariant
morphism''.  That is, for each pair of objects $x$ and $y$, there are
two hom-categories $\A\p(x,y)$ and $\A\m(x,y)$, whose objects we call
\emph{covariant} and \emph{contravariant} morphisms respectively.
Composition is defined in the obvious way: the composite of two
morphisms of the same variance is covariant, while the composite of
two morphisms of different variances is contravariant.  In addition,
postcomposing with a contravariant morphism is contravariant on
2-cells.  We call such a gadget a \textbf{2-category with
  contravariance}.

As with any sort of generalized multicategory, we can characterize the
virtual $T$-algebras that are pseudo $T$-algebras by a notion of
\emph{representability}.  This means that for each object $x$, we have
an object $x\o$ and isomorphisms $\A\m(x,y) \cong \A\p(x\o,y)$ and
$\A\p(x,y) \cong \A\m(x\o,y)$, jointly natural in $y$.  We call an
object $x\o$ with this property a \emph{(strict) opposite} of $x$.
The corresponding pseudo $T$-algebra structure describes this
operation $(-)\o$ as a \emph{strong} duality involution on the
underlying 2-category $\A\p$, meaning a strict 2-functor
$(\A\p)\co \to \A\p$ that is self-inverse up to coherent strict
2-natural isomorphism.

Now, it turns out that 2-categories with contravariance are not just
generalized multicategories: they are also \emph{enriched
  categories}.\footnote{Representation of additional structure on a
  category as enrichment occurs in many other places; see for
  instance~\cite{gp:enr-var,ls:limlax,garner:F,gp:enr-lawvere,garner:emb-tancat}.}
Namely, there is a (non-symmetric) monoidal category, denoted \V (for
Variance), such that \V-enriched categories are the same as
2-categories with contravariance.  (As a category, \V is just
$\bCat\times\bCat$, but its monoidal structure is not the usual one.)
From this perspective, we can alternatively describe strict opposites
as \emph{weighted colimits}: $x\o$ is the copower (or ``tensor'') of
$x$ by a particular object $\du$ of \V, called the \emph{dual unit}.
Since $\du$ is dualizable in \V, opposites are an \emph{absolute} or
\emph{Cauchy} colimit in the sense of~\cite{street:absolute}: they are
preserved by all \V-enriched functors.  It follows that any
2-category-with-contravariance has a ``completion'' with respect to
opposites, and this operation is idempotent.

We have now moved into a context having a straightforward
bicategorical version.  We simply observe that \V can be made into a
monoidal 2-category, and consider \emph{\V-enriched bicategories}; we
call these \emph{bicategories with contravariance}.  In such a
bicategory we can consider ``weak opposites'', asking only for
pseudonatural equivalences $\A\m(x,y) \simeq \A\p(x\o,y)$ and
$\A\p(x,y) \simeq \A\m(x\o,y)$; these are ``absolute weighted
bicolimits'' in the sense of~\cite{gs:freecocomp}.  Since any
isomorphism of categories is an equivalence, any strict opposite is
also a weak one.  (More abstractly, strict opposites should be
\emph{flexible colimits}~\cite{bkps:flexible} in a suitable sense, but
we will not make this precise.)

Now, it is straightforward to generalize the coherence theorem for
bicategories to a coherence theorem for \emph{enriched} bicategories.
Therefore, any bicategory with contravariance is biequivalent to a
2-category with contravariance.  This suggests that the process by
which we arrived at \V-enriched categories could be duplicated on the
bicategorical side, yielding the following ``ladder'' strategy for
proving \autoref{thm:intro}:
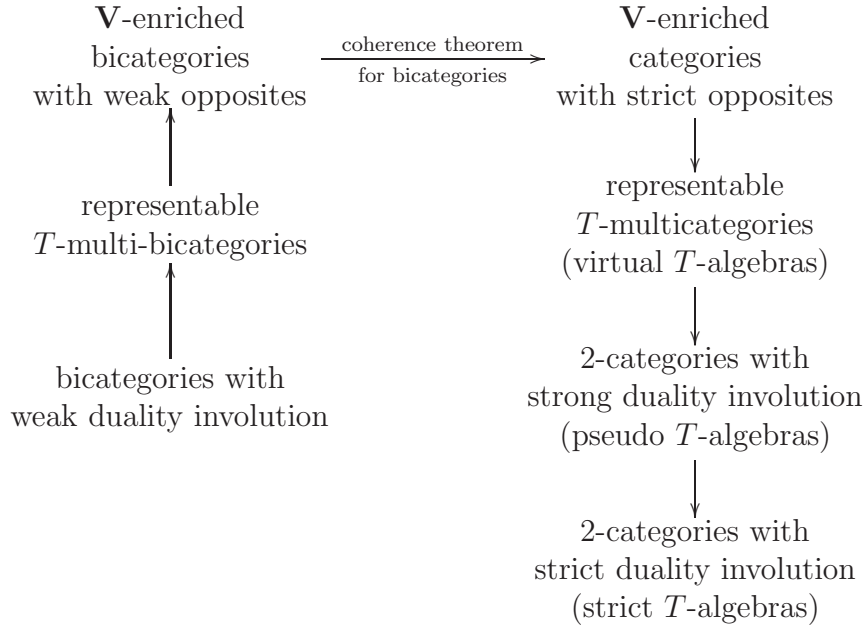
\begin{equation}
  \vcenter{\xymatrix@C=4pc{
      \text{\parbox{1.5in}{\centering \V-enriched bicategories\\with weak opposites}}\ar[r]^{\text{coherence theorem}}_{\text{for bicategories}} &
      \text{\parbox{1.5in}{\centering \V-enriched categories\\with strict opposites}}\ar[d]\\
      \text{\parbox{2in}{\centering representable\\$T$-multi-bicategories}}\ar[u] &
      \text{\parbox{2in}{\centering representable\\$T$-multicategories\\(virtual $T$-algebras)}}\ar[d]\\
      \text{\parbox{2in}{\centering bicategories with\\weak duality involution}}\ar[u] &
      \text{\parbox{2in}{\centering 2-categories with\\strong duality involution\\(pseudo $T$-algebras)}}\ar[d]\\
      & \text{\parbox{2in}{\centering 2-categories with\\strict duality involution\\(strict $T$-algebras)}}
    }}
\end{equation}
There are three problems with this idea, two minor and one major.  The
first is that it (apparently) produces only a strong duality
involution rather than a strict one, necessitating an extra step at
the bottom-right of the ladder, as shown.  However, the
strictification of pseudo-algebras for 2-monads is fairly
well-understood, so we can apply a general coherence
result~\cite{power:coherence,lack:codescent-coh}.

The second problem is that \textit{a priori}, the coherence theorem
for \V-enriched bicategories does not also strictify the weak
opposites into strict opposites.  However, this is also easy to
remedy: since the strictification of a \V-bicategory with weak
opposites will still have weak opposites, and any strict opposite is
also a weak one, it will be biequivalent to its free cocompletion
under strict opposites.

The third, and more major, problem with this strategy is that there is
no extant theory of ``generalized multi-bicategories''.  We could
develop such a theory, but it would take us rather far afield.  Thus,
instead we will ``hop over'' that rung of the ladder by constructing a
$\V$-enriched bicategory with weak opposites directly from a
bicategory with a weak duality involution, by a ``beta-reduced'' and
weakened version of the analogous operation on the other side.

Since this direct construction also includes the strict case, we
could, formally speaking, dispense with the multicategories on the
other side as well.  Indeed, the entire proof can be beta-reduced into
a more compact form: if we prove the coherence theorem for enriched
bicategories using a Yoneda embedding, the strictification and
cocompletion processes could be combined into one and tweaked slightly
to give a strict duality involution directly.

In fact, there are not many applications of \autoref{thm:intro}
anyway.  First of all, it is not all that easy to think of naturally
occurring duality involutions that are not already strict.  But here
are a few:
\begin{enumerate}[leftmargin=2em,label=(\arabic*)]
\item The 2-category of fibrations over some base category \bS has a
``fiberwise'' duality involution, but since its action on non-vertical
arrows has to be constructed in a more complicated way than simply
turning them around, it is not strict.
\item If \cB is a compact closed
bicategory~\cite{ds:monbi-hopfagbd,stay:ccb}, then its bicategory
$\cMap(\cB)$ of maps (left adjoints) has a duality involution that is
not generally strict.
\item If \cA is a bicategory with a duality involution, and \cW is a
class of morphisms in \cA admitting a calculus of
fractions~\cite{pronk:bicat-frac} and closed under the duality
involution, then the bicategory of fractions \awinv inherits a duality
involution that is not strict (even if the one on \cA was strict).
\end{enumerate} However, even in these cases \autoref{thm:intro} is
not as important as it might be, because Lack's coherence theorem
(``naturally occurring bicategories are biequivalent to naturally
occurring 2-categories'') applies very strongly to duality
involutions: nearly all naturally occurring bicategories with duality
involutions are biequivalent to some \emph{naturally occurring} strict
2-category with a strict duality involution.  For the examples above,
we have:
\begin{enumerate}[leftmargin=2em,label=(\arabic*)]
\item The 2-category of fibrations over \bS is biequivalent to the
2-category of \bS-indexed categories, which has a strict duality
involution inherited from \cCat.)
\item For the standard examples of compact closed bicategories such as
\cProf or \cSpan, the bicategory of maps is biequivalent to a
well-known strict 2-category with a strict duality involution, such as
$\cCat_{\mathrm{cc}}$ (Cauchy-complete categories) or \bSet.
\item Many naturally occurring examples of bicategories of fractions
are also biequivalent to well-known 2-categories with strict duality
involutions, such as some 2-category of stacks.
\end{enumerate} Thus, if \autoref{thm:intro} were the main point of
this paper, it would be somewhat disappointing.  However, I regard the
method of proof, and the entire ladder it gives rise to, as more
important than the result itself.  Representating contravariance using
generalized multicategories and enrichment seems a promising avenue
for future study of further properties of duality involutions.  From
this perspective, the paper is primarily a contribution to
\emph{enhanced 2-category theory} in the sense of~\cite{ls:limlax},
which just happens to prove a coherence theorem to illustrate the
ideas.

Furthermore, our abstract approach also generalizes to other types of
contravariance.  The right-hand side of the ladder, at least, works in
the generality of any group action on any monoidal category \bW.  The
motivating case of duality involutions on 2-categories is the case
when $\dZ/2\dZ$ acts on \bCat by $(-)\op$; but other actions
representing other kinds of contravariance include the following.
\begin{itemize}[leftmargin=2em]
\item $\dZ/2\dZ \times \dZ/2\dZ$ acts on $\mathbf{2Cat}$ by $(-)\op$
and $(-)\co$.  When $\mathbf{2Cat}$ is given the Gray monoidal
structure, this yields a theory of duality involutions on
Gray-categories.
\item $(\dZ/2\dZ)^n$ acts on strict $n$-categories (including the case
$n=\omega$), yielding duality involutions for strict
$(n+1)$-categories.  
not as interesting as weak ones, but their theory can point the way
towards a weak version.
\item $\dZ/2\dZ$ acts on the category $\mathbf{sSet}$ of simplicial
sets by reversing the directions of all the simplices.  With
simplicial sets modeling $(\infty,1)$-categories as quasicategories,
this yields a theory of duality involutions on a particular model for
$(\infty,2)$-categories (see for instance~\cite{rv:fibyon-oocosmos}).
\item Combining the ideas of the last two examples, $(\dZ/2\dZ)^n$
acts on the category of $\Theta_n$-spaces by reversing direction at
all dimensions, leading to duality involutions on an enriched-category
model for $(\infty,n+1)$-categories~\cite{br:cmp-infn-i}.
\end{itemize} We will not develop any of these examples further here,
but the perspective of describing contravariance through enrichment
may be useful for all of them as well.

We begin in \cref{sec:strong-duality} by defining weak, strong, and
strict duality involutions.  Then we proceed up the ladder from the
bottom right.  In \cref{sec:2-monadic-approach} we express strong and
strict duality involutions as algebra structures for a 2-monad, and
deduce that strong ones can be strictified.  In \cref{sec:genmulti} we
express strong duality involutions using generalized multicategories,
and in sections \ref{sec:contr-thro-enrichm}--\ref{sec:opposites} we
reexpress them using enrichment.  In \cref{sec:bicat-contra} we jump
over to the other side of the ladder, showing that weak duality
involutions on bicategories can be expressed using bicategorical
enrichment.  Then finally in \cref{sec:coherence} we cross the top of
the ladder with a coherence theorem for enriched bicategories.

\section{Duality involutions}
\label{sec:strong-duality}

In this section we define strict, strong, and weak duality
involutions, allowing us to state \autoref{thm:intro} precisely.

\begin{defn}\label{defn:duality-involution}
  A \textbf{weak duality involution} on a bicategory \cA consists of:
  \begin{itemize}[leftmargin=2em]
  \item A pseudofunctor $(-)\o : \cA\co \too \cA$.
  \item A pseudonatural adjoint equivalence
    \[ \xymatrix{
      \cA \ar[dr]_{((-)\o)\co} \ar@{=}[rr] && \cA.\\
      & \cA\co \ar[ur]_{(-)\o} \ar@{}[u]|(.6){\Downarrow\fy}
    }
    \]
  \item An invertible modification
    \[ \vcenter{\xymatrix{
        \cA\co \ar[r]^{(-)\o} & \cA \ar[dr]_{((-)\o)\co} \ar@{=}[rr] && \cA\\
        && \cA\co \ar[ur]_{(-)\o} \ar@{}[u]|(.6){\Downarrow\fy}
      }}
    \quad\overset{\ze}{\Longrightarrow}\quad
    \vcenter{\xymatrix{
        \cA\co \ar[dr]_{(-)\o} \ar@{=}[rr] && \cA\co \ar[r]^{(-)\o} & \cA\\
        & \cA \ar[ur]_{((-)\o)\co} \ar@{}[u]|(.6){\Downarrow\fy\co}
      }}
    \]
  whose components are therefore 2-cells $\ze_x: \fy_{x\o} \toiso (\fy_x)\o$.
  \item For any $x\in\cA$, we have
    \[\vcenter{\xymatrix{
        x \ar[r]^{\fy_x} &
        x\oo
        \ar@(ur,ul)[rr]^{\fy_{x\oo}}
        \ar@(dr,dl)[rr]_{(\fy_{x\o})\o}
        \ar@{}[rr]|{\Downarrow \ze_{x\o}} &&
        x\oooo
      }}
    \quad=\quad
    \vcenter{\xymatrix{
        & x\oo \ar@(r,u)[ddrr]^{\fy_{x\oo}}\\
        & {\scriptstyle \Downarrow\cong}\\
        x \ar[r]^{\fy_x} \ar[uur]^{\fy_x} &
        x\oo
        \ar@(ur,ul)[rr]^{(\fy_x)\oo}
        \ar@(dr,dl)[rr]_{(\fy_{x\o})\o}
        \ar@{}[rr]|{\Downarrow \ze_x\o} &&
        x\ooo{}\o
      }}
    \]
    (the unnamed isomorphism is a pseudonaturality constraint for \fy).
  \end{itemize}
  If \cA is a strict 2-category, a \textbf{strong duality involution}
  on \cA is a weak duality involution for which
  \begin{itemize}[leftmargin=2em]
  \item $(-)\o$ is a strict 2-functor,
  \item $\fy$ is a strict 2-natural isomorphism, and
  \item $\ze$ is an identity.
  \end{itemize}
  If moreover $\fy$ is an identity, we call it a \textbf{strict duality involution}.
\end{defn}

In particular, \fy and $\ze$ in a weak duality involution exhibit
$(-)\o$ and $((-)\o)\co$ as a biadjoint biequivalence between \cA and
$\cA\co$, in the sense of~\cite{gurski:biequiv}.  Similarly, in a
strong duality involution, \fy exhibits $(-)\o$ and $((-)\o)\co$ as a
2-adjoint 2-equivalence between \cA and $\cA\co$.  And, of course, in
a strict duality involution, $(-)\o$ and $((-)\o)\co$ are inverse
isomorphisms of 2-categories.

\begin{defn}\label{defn:duality-functor}
  If \cA and \cB are bicategories equipped with weak duality
  involutions, a \textbf{duality pseudofunctor} $F:\cA\to\cB$ is a
  pseudofunctor equipped with
  \begin{itemize}[leftmargin=2em]
  \item A pseudonatural adjoint equivalence
    \[\vcenter{\xymatrix{
        \cA\co\ar[d]_{(-)\o}\ar[r]^-{F\co} \drtwocell\omit{\fii} &
        \cB\co\ar[d]^{(-)\o}\\
        \cA\ar[r]_-F &
        \cB.
      }}\]
  \item An invertible modification
    \[\vcenter{\xymatrix@R=3pc@C=3pc{
        && \cA \ar[dl]_{((-)\o)\co} \ar[r]^F \dtwocell\omit{\fii}
        & \cB \ar[dl]|{((-)\o)\co} \ddluppertwocell^{1_\cB}{\fy}\\
        &\cA\co\ar[d]_{(-)\o}\ar[r]|-{F\co} \drtwocell\omit{\fii} &
        \cB\co\ar[d]|{(-)\o}\\
        &\cA\ar[r]_-F &
        \cB
      }}
    \quad \overset{\theta}{\Longrightarrow}\quad
    \vcenter{\xymatrix@R=3pc@C=3pc{
        & \cA \ar[dl]|{((-)\o)\co} \ddluppertwocell^{1_\cA}{\fy}\\
        \cA\co\ar[d]|{(-)\o}\\
        \cA
      }}
    \]
    whose components are therefore 2-cells in \cB of the following shape:
    \[\vcenter{\xymatrix{
        (F x)\oo\ar[r]^{(\fii_x)\o} \ar@{}[dr]|{\Downarrow \th_x} &
        (F (x\o))\o\ar[d]^{\fii_{x\o}}\\
        F x\ar[u]^{\fy_{F x}}\ar[r]_{F(\fy_x)} &
        F(x\oo).
      }}\]
  \item For any $x\in\cA$, we have
    \begin{multline}
      \vcenter{\xymatrix@R=3pc@C=1.5pc{
          (Fx)\ooo\ar[rr]^{(\fii_x)\oo} \drrtwocell\omit{\mathrlap{(\theta_x)\o}} &&
          (F(x\o))\oo\ar[d]^{(\fii_{x\o})\o}\\
          (Fx)\o\ar[rr]_{(F\fy_x)\o} \ar[d]_{\fii_x} \drrtwocell\omit{\cong}
          \utwocell^{\mathllap{\fy_{(F x)\o}}}_{\mathrlap{(\fy_{F x})\o}}{\ze_{F x}} &&
          (F(x\oo))\o \ar[d]^{\fii_{x\oo}}\\
          F(x\o) \ar[rr]_{F((\fy_x)\o)} && F(x\ooo)
        } }
      =
      \vcenter{\xymatrix@R=4pc@C=1.5pc{
          (Fx)\ooo \ar[r]^{(\fii_x)\oo} &
          (F(x\o))\oo \ar[rr]^{(\fii_{x\o})\o} \drrtwocell\omit{\mathrlap{\theta_{x\o}}} &&
          (F(x\oo))\o \ar[d]^{\fii_{x\oo}}\\
          (Fx)\o \ar[u]^{\fy_{(Fx)\o}} \ar[r]_{\fii_x} \urtwocell\omit{\cong} &
          F(x\o) \ar[u]|{\fy_{F(x\o)}} \rrtwocell^{F(\fy_{x\o})}_{F((\fy_x)\o)}{\mathrlap{F(\ze_x)}} &&
          F(x\ooo)
        }}
    \end{multline}
    (the unnamed isomorphisms are pseudonaturality constraints for $\fii$ and $\fy$).
  \end{itemize}
  If \cA and \cB are strict 2-categories with strong duality
  involutions, then a \textbf{(strong) duality 2-functor}
  $F:\cA\to\cB$ is a duality pseudofunctor such that
  \begin{itemize}[leftmargin=2em]
  \item $F$ is a strict 2-functor,
  \item $\fii$ is a strict 2-natural isomorphism, and
  \item $\th$ is an identity.
  \end{itemize}
  If $\fii$ is also an identity, we call it a \textbf{strict duality 2-functor}.
\end{defn}

Note that if the duality involutions of \cA and \cB are strict, then
the identity $\th$ says that $(\fii_x)\o = (\fii_{x\o})^{-1}$.  On the
other hand, if $\cA$ is a strict 2-category with two strong duality
involutions $(-)\o$ and $(-)\oprime$, to make the identity 2-functor
into a duality 2-functor is to give a natural isomorphism
$A\o \cong A\oprime$ that commutes with the isomorphisms $\fy$ and
$\fy'$.

Now \autoref{thm:intro} can be stated more precisely as:

\begin{thm}\label{thm:main}
  If $\cA$ is a bicategory with a weak duality involution, then there
  is a 2-category $\cA'$ with a strict duality involution and a
  duality pseudofunctor $\cA\to\cA'$ that is a biequivalence.
\end{thm}

We could make this more algebraic by defining a whole tricategory of
bicategories with weak duality involution and showing that our
biequivalence lifts to an internal biequivalence therein, but we leave
that to the interested reader.  In fact, the correct definitions of
transformations and modifications can be extracted from our
characterization via enrichment.  (It does turn out that there is no
obvious way to define non-invertible modifications.)

\begin{rmk}
  The definition of duality involution may seem a little \textit{ad
    hoc}.  In \cref{sec:bicat-contra} we will rephrase it as a special
  case of a ``twisted group action'', which may make it seem more
  natural.
\end{rmk}

We end this section with some examples.

\begin{eg}
  With nearly any reasonable set-theoretic definition of ``category''
  and ``opposite'', the 2-category $\cCat$ of categories and functors
  has a strict duality involution.  The same is true for the
  2-category of categories enriched over any symmetric monoidal
  category, or the 2-category of categories internal to some category
  with pullbacks.
\end{eg}

\begin{eg}
  If \cA is a bicategory with a weak duality involution and \cK is a
  locally groupoidal bicategory, then the bicategory $[\cK,\cA]$ of
  pseudofunctors, pseudonatural transformations, and modifications
  inherits a weak duality involution by applying the duality
  involution of \cA pointwise.  Local groupoidalness of \cK ensures
  that $\cK\cong\cK\co$, so that we can define the dual of a
  pseudofunctor $F:\cK\to\cA$ to be
  \[F\o : \cK \cong \cK\co \xto{F\co} \cA\co \xto{(-)\o} \cA.\]
  The rest of the structure follows by whiskering.  If \cA is a
  2-category and its involution is strong or strict, the same is true
  for $[\cK,\cA]$.
\end{eg}

\begin{eg}
  If $\cA$ is a bicategory with a weak duality involution and
  $F:\cA\to\cC$ is a biequivalence, then $\cC$ can be given a weak
  duality involution making $F$ a duality pseudofunctor.  We first
  have to enhance $F$ to a biadjoint biequivalence as
  in~\cite{gurski:biequiv}; then we define all the structure by
  composing with $F$ and its inverse.
\end{eg}

\begin{eg}
  The 2-category of fibrations over a base category \bS has a strong
  duality involution constructed as follows.  Given a fibration
  $P:\bC\to\bS$, in its dual $P\o:\bC\o\to\bS$ the objects of $\bC\o$
  are those of \bC, while the morphisms from $x$ to $y$ over a
  morphism $f:a\to b$ in \bS are the morphisms $f^*y \to x$ over $a$
  in $\bC$.  Here $f^*y$ denotes the pullback of $y$ along $f$
  obtained from some cartesian lifting; the resulting ``set of
  morphisms from $x$ to $y$'' in $\bC\o$ is independent, up to
  canonical isomorphism, of the choice of cartesian lift.  However,
  there is no obvious way to define it such that $\bC\oo$ is
  \emph{equal} to $\bC$, rather than merely canonically isomorphic.
  Of course, the 2-category of fibrations over \bS is biequivalent to
  the 2-category of \bS-indexed categories, which has a strict duality
  involution induced from its codomain \cCat.
\end{eg}

\begin{eg}
  Let \cA be a bicategory with a duality involution, let \cW be a
  class of morphisms of \cA admitting a \emph{calculus of right
    fractions} in the sense of~\cite{pronk:bicat-frac}, and suppose
  moreover that if $v\in\cW$ then $v\o\in\cW$.  Then the bicategory of
  fractions \awinv also admits a duality involution, constructed using
  its universal property~\cite[Theorem 21]{pronk:bicat-frac} as
  follows.

  Let $\ell : \cA\to\awinv$ be the localization functor.  By
  assumption, the composite
  $\cA\xto{((-)\o)\co}\cA\co \xto{\ell\co} \awinv\co$ takes morphisms
  in \cW to equivalences.  Thus, it factors through $\ell$, up to
  equivalence, by a functor that we denote
  $((-)\ow)\co: \awinv \to \awinv\co$ (that is, a functor whose 2-cell
  dual we denote $(-)\ow$).
  Now the pasting composite composite
  \[\xymatrix{
      & && \cA \ar[dr]^\ell\\
      \cA \ar[dr]_{\ell} \ar[rr]_{((-)\o)\co} \ar@{=}@/^5mm/[urrr] \ar@{}[drrr]|{\Downarrow\simeq} \ar@{}[urrr]|{\Downarrow\fy} &&
      \cA\co \ar[dr]_{\ell\co} \ar[ur]_{(-)\o} \ar@{}[rr]|{\Downarrow\simeq} && \awinv\\
      & \awinv \ar[rr]_{((-)\ow)\co} && \awinv\co \ar[ur]_{(-)\ow}
    }\]
    is a pseudonatural equivalence from $\ell$ to
    $(-)\ow \circ ((-)\ow)\co \circ \ell$.  Hence, by the universal
    property of $\ell$, it is isomorphic to the whiskering by $\ell$
    of some pseudonatural equivalence
  \[ \xymatrix{
      \awinv \ar[dr]_{((-)\ow)\co} \ar@{=}[rr] && \awinv.\\
      & \awinv\co \ar[ur]_{(-)\ow} \ar@{}[u]|(.6){\Downarrow\fy'}
    }
  \]
  Similar whiskering arguments produce the modification $\zeta'$ and verify its axiom.

  Note that this induced duality involution on \awinv will not
  generally be strict, even if the one on \cA is.  Specifically, with
  careful choices we can make $(-)\ow$ strictly involutory on objects,
  1-cells, and 2-cells, but there is no obvious way to make it a
  strict 2-functor.  (On the other hand, as remarked in
  \cref{sec:introduction}, often \awinv is biequivalent to some
  naturally-occurring 2-category having a strict duality involution,
  such as the examples of \'etendues and stacks considered
  in~\cite{pronk:bicat-frac}.)

  A related special case is that if we work in an ambient set theory
  not assuming the axiom of choice, then we might take $\cA=\cCat$ and
  \cW the class of fully faithful and essentially surjective functors.
  In this case \awinv is equivalent to the bicategory of categories
  and \emph{anafunctors}~\cite{makkai:avoiding-choice,roberts:ana},
  which therefore inherits a weak duality involution.
\end{eg}

\begin{eg}
  Let \cB be a \emph{compact closed bicategory} (also called
  \emph{symmetric autonomous}) as
  in~\cite{ds:monbi-hopfagbd,stay:ccb}.  Thus means that \cB is
  symmetric monoidal, and moreover each object $x$ has a dual $x\o$
  with respect to the monoidal structure, with morphisms
  $\eta:\unit \to x\otimes x\o$ and $\ep : x\o\otimes x \to \unit$
  satisfying the triangle identities up to isomorphism.  If we choose
  such a dual for each object, then $(-)\o$ can be made into a
  biequivalence $\cB\op \to \cB$, sending a morphism $g:y\to x$ to the
  composite
  \begin{equation}
    \xymatrix{ x\o \ar[r]^-{\eta_y} & x\o\otimes y\otimes y\o \ar[r]^-{g} & x\o\otimes x \otimes y\o \ar[r]^-{\ep_x} & y\o},\label{eq:gop}
  \end{equation}
  with $\eta$ and $\ep$ becoming pseudonatural transformations.
  Moreover, this functor $\cB\op \to \cB$ looks exactly like a duality
  involution except that $(-)\co$ has been replaced by $(-)\op$: we
  have a pseudonatural adjoint equivalence
  \[ \xymatrix{
      \cB \ar[dr]_{((-)\o)\op} \ar@{=}[rr] && \cB.\\
      & \cB\op \ar[ur]_{(-)\o} \ar@{}[u]|(.6){\Downarrow\fy}
    }
  \]
  and an invertible modification
  \[ \vcenter{\xymatrix{
        \cB\op \ar[r]^{(-)\o} & \cB \ar[dr]_{((-)\o)\op} \ar@{=}[rr] && \cB\\
        && \cB\op \ar[ur]_{(-)\o} \ar@{}[u]|(.6){\Downarrow\fy}
      }}
    \quad\overset{\ze}{\Longrightarrow}\quad
    \vcenter{\xymatrix{
        \cB\op \ar[dr]_{(-)\o} \ar@{=}[rr] && \cB\op \ar[r]^{(-)\o} & \cB\\
        & \cB \ar[ur]_{((-)\o)\op} \ar@{}[u]|(.6){\Downarrow\fy\op}
      }}
  \]
  satisfying the same axiom as in \cref{defn:duality-involution}.
  Explicitly, $\fy$ is the composite
  \[ x \xto{x\otimes \eta_{x\o}} x\otimes x\o \otimes x\oo \toiso x\o \otimes x\otimes x\oo \xto{\ep_x\otimes x\oo} x\oo \]
  and $\zeta$ is obtained as a pasting composite
  \[
    \xymatrix{ & x\o x\oo x\ooo \ar[r]^{\sim} \ar[dd]  & x\oo x\o x\ooo \ar[dr]^{\ep_{x\o} x\ooo} \ar[dd] \\
      x\o \ar[ur]^{x\o\eta_{x\oo}} \ar[dr]_{(x \eta_{x\o})\o}
      \ar@{}[r]|{\Downarrow\cong} & \ar@{}[r]|{\Downarrow\cong}& \ar@{}[r]|{\Downarrow\cong} &
      x\ooo\\
      &(x x\o x\oo)\o \ar[r]_\sim & (x\o x x\oo)\o \ar[ur]_{(\ep_x x\oo)\o}
    }
  \]
  using the fact that if $x\o$ is a dual of $x$, then by symmetry of \cB, $x$ is a dual of $x\o$.

  Now let \cA be the locally full sub-bicategory of maps (left
  adjoints) in \cB.  Since passing from left to right adjoints
  reverses the direction of 2-cells, we have a ``take the right
  adjoint'' functor $\cA\coop \to \cB$, or equivalently
  $\cA\co \to \cB\op$.  Composing with the above ``duality'' functor
  $\cB\op\to\cB$, we have a functor $\cA\co \to\cB$, and since right
  adjoints in $\cB$ are left adjoints in $\cB\op$, this functor lands
  in \cA, giving $(-)\o:\cA\co\to\cA$.  Of course, equivalences are
  maps, so the above $\fy$ and $\zeta$ lie in \cA, and therefore equip
  $\cA$ with a duality involution.

  This duality involution on \cA is not generally strict or even
  strong.  However, as remarked in \cref{sec:introduction}, in many
  naturally-ocurring examples \cA is equivalent to some
  naturally-ocurring 2-category with a strict duality involution.  For
  instance, if $\cB=\cProf$ then $\cA\simeq \cCat_{\mathrm{cc}}$, the
  2-category of Cauchy-complete categories; while if $\cB=\cSpan$ then
  $\cA\simeq \bSet$, and similarly for internal and enriched versions.
\end{eg}

\section{A 2-monadic approach}
\label{sec:2-monadic-approach}

Let $T(\cA) = \cA + \cA\co$, an endo-2-functor of the 2-category
\twocat of 2-categories, strict 2-functors, and strict 2-natural
transformations.  We have an obvious strict 2-natural transformation
$\eta:\mathsf{Id}\to T$, and we define $\mu:TT\to T$ by
\[ (\cA + \cA\co) + (\cA + \cA\co)\co \toiso \cA + \cA\co + \cA\co + \cA \xto{\nabla} \cA + \cA\co \]
where $\nabla$ is the obvious ``fold'' map.

\begin{thm}\label{thm:2monad}
  $T$ is a strict 2-monad, and:
  \begin{enumerate}
  \item Normal pseudo $T$-algebras are 2-categories with strong duality involutions;\label{item:2m1}
  \item Pseudo $T$-morphisms are duality 2-functors; and\label{item:2m2}
  \item Strict $T$-algebras are 2-categories with strict duality involutions.\label{item:2m3}
  \end{enumerate}
\end{thm}
\begin{proof}
  The 2-monad laws for $T$ are straightforward to check.  By a normal
  pseudo algebra we mean one for which the unit constraint identifying
  $\cA\to T\cA \to \cA$ with the identity map is itself an identity.
  Thus, when $T\cA = \cA +\cA\co$, this means the action
  $a:T\cA\to\cA$ contains no data beyond a 2-functor
  $(-)\o:\cA\co\to\cA$.  The remaining data is a 2-natural isomorphism
  \begin{equation}
    \vcenter{\xymatrix{
        TT\cA\ar[r]^{Ta}\ar[d]_{\mu} \drtwocell\omit{\cong} &
        T\cA\ar[d]^{a}\\
        T\cA\ar[r]_{a} &
        \cA
      }}
    \quad\text{that is}\quad
    \vcenter{\xymatrix{
        \cA+\cA\co+\cA\co+\cA\ar[r]\ar[d] \drtwocell\omit{\cong} &
        \cA+\cA\co\ar[d]\\
        \cA+\cA\co\ar[r] &
        \cA
      }}
  \end{equation}
  satisfying three axioms that can be found, for instance,
  in~\cite[\S1]{lack:codescent-coh}.  The right-hand square commutes
  strictly on the first three summands in its domain, and the second
  and third of the coherence axioms say exactly that the given
  isomorphism in these cases is an identity.  Thus, what remains is
  the component of the isomorphism on the fourth summand, which has
  precisely the form of $\fy$ in \cref{defn:duality-involution}, and
  it is easy to check that the first coherence axiom reduces in this
  case to the identity $\zeta$.  This proves~\ref{item:2m1},
  and~\ref{item:2m3} follows immediately.

  Similarly, for~\ref{item:2m2}, a pseudo $T$-morphism is a 2-functor
  $F:\cA\to\cB$ together with a 2-natural isomorphism
  \begin{equation}
    \vcenter{\xymatrix{
        T\cA\ar[r]^{T F}\ar[d] \drtwocell\omit{\cong} &
        T\cB\ar[d]\\
        \cA\ar[r]_F &
        \cB
      }}
    \quad\text{that is}\quad
    \vcenter{\xymatrix{
        \cA+\cA\co\ar[r]^{F+F\co}\ar[d] \drtwocell\omit{\cong} &
        \cB+\cA\co\ar[d]\\
        \cA\ar[r]_F &
        \cB
      }}
  \end{equation}
  satisfying two coherence axioms also listed
  in~\cite[\S1]{lack:codescent-coh}.  This square commutes strictly on
  the first summand of its domain, and the second coherence axiom
  ensures that the isomorphism is the identity there.  So the
  remaining data is the isomorphism on the second summand, which has
  precisely the form of $\fii$ in \cref{defn:duality-functor}, and the
  first coherence axiom reduces to the identity $\theta$.
\end{proof}

In particular, we obtain an automatic definition of a ``duality
2-natural transformation'': a $T$-2-cell between pseudo $T$-morphisms.
This also gives us another source of examples.

\begin{eg}
  The 2-category $T\algs$ of strict $T$-algebras and strict
  $T$-morphisms is complete with limits created in $\twocat$,
  including in particular Eilenberg--Moore objects~\cite{street:ftm}.
  Thus, for any monad $M$ in this 2-category --- which is to say, a
  2-monad that is a strict duality 2-functor and whose unit and
  multiplication are duality 2-natural transformations --- the
  2-category $M\algs$ of strict $M$-algebras and strict $M$-morphisms
  is again a strict $T$-algebra, i.e.\ has a strict duality
  involution.

  Similarly, by~\cite{bkp:2dmonads} the 2-category $T\alg$ of strict
  $T$-algebras and pseudo $T$-morphisms has PIE-limits, including
  EM-objects.  Thus, we can reach the same conclusion even if $M$ is
  only a strong duality 2-functor.  And since \twocat is locally
  presentable and $T$ has a rank, there is another 2-monad $T'$ whose
  strict algebras are the pseudo $T$-algebras; thus we can argue
  similarly in the 2-category $T\psalg$ of pseudo $T$-algebras and
  pseudo $T$-morphisms, so that strong duality involutions also lift
  to $M\algs$.

  Usually, of course, we are more interested in the 2-category $M\alg$
  of strict $M$-algebras and \emph{pseudo} $M$-morphisms.  It might be
  possible to enhance the above abstract argument to apply to this
  case using techniques such as~\cite{lack:psmonads,power:3dmonads},
  but it is easy enough to check directly that if $M$ lies in $T\alg$
  or $T\psalg$, then so does $M\alg$.  If $(A,a)$ is an $M$-algebra,
  then the induced $M$-algebra structure on $A\o$ is the composite
  \[ M(A\o) \xto{\fii} (MA)\o \xto{a\o} A\o \]
  and if $(f,\fbar):(A,a) \to (B,b)$ is a pseudo $M$-morphism (where
  $\fbar : a\circ Mf \toiso f \circ b$), then $f\o$ becomes a pseudo
  $M$-morphism with the following structure 2-cell:
  \begin{equation}
    \vcenter{\xymatrix{
        M(B\o)\ar[r]^{\fii} &
        (MB)\o\ar[rr]^{b\o} &&
        B\o\\
        M(A\o)\ar[r]_{\fii}\ar[u]^{M(f\o)} \urtwocell\omit{\cong} &
        (MA)\o\ar[rr]_{a\o}\ar[u]|{(Mf)\o} \urrtwocell\omit{\mathrlap{(\fbar^{-1})\o}} &&
        A\o\ar[u]_{f\o}
      }}
  \end{equation}
  The axiom $\theta$ of $\fii$ (which is an equality since $M$ is a
  strong duality 2-functor) ensures that $\fy$ lifts to $M\alg$
  (indeed, to $M\algs$), and its own $\theta$ axiom is automatic.  A
  similar argument applies to $M\psalg$.  Thus, 2-categories of
  algebraically structured categories such as monoidal categories,
  braided or symmetric monoidal categories, and so on, admit strict
  duality involutions, even when their morphisms are of the pseudo
  sort.  (Of course, this is impossible for lax or colax morphisms,
  since dualizing the categories involved would switch lax with
  colax.)

  In theory, this could be another source of weak duality involutions
  that are not strong: if for a 2-monad $M$ the transformation $\fii$
  were not a strictly 2-natural isomorphism or its axiom $\theta$ were
  not strict, then $M\alg$ would only inherit a weak duality
  involution, even if the duality involution on the original
  2-category were strict.  However, I do not know any examples of
  2-monads that behave this way.
\end{eg}

We end this section with the strong-to-strict coherence theorem.

\begin{thm}\label{thm:2monad-coherence}
  If \cA is a 2-category with a strong duality involution, then there
  is a 2-category $\cA'$ with a strict duality involution and a
  duality 2-functor $\cA\to\cA'$ that is a 2-equivalence.
\end{thm}
\begin{proof}
  The 2-category $\twocat$ admits a factorization system $(\cE,\cM)$
  in which $\cE$ consists of the 2-functors that are bijective on
  objects and $\cM$ of the 2-functors that are 2-fully-faithful, i.e.\
  an isomorphism on hom-categories.  Moreover, this factorization
  system satisfies the assumptions of~\cite[Theorem
  4.10]{lack:codescent-coh}, and we have $T\cE\subseteq \cE$.
  Thus,~\cite[Theorem 4.10]{lack:codescent-coh} (which is an abstract
  version of~\cite[Theorem 3.4]{power:coherence}), together with the
  characterizations of \cref{thm:2monad}, implies the desired result.
\end{proof}

Inspecting the proof of the general coherence theorem, we obtain a
concrete construction of $\cA'$: it is the result of factoring the
pseudo-action map $T\cA \to \cA$ as a bijective-on-objects 2-functor
followed by a 2-fully-faithful one.  In other words, the objects of
$\cA'$ are two copies of the objects of $\cA$, one copy representing
each object and one its opposite, with the duality interchanging them.
The morphisms and 2-cells are then easy to determine.

It remains, therefore, to pass from a weak duality involution on a
bicategory to a strong one on a 2-category.  We proceed up the
right-hand side of the ladder from \cref{sec:introduction}.

\section{Contravariance through virtualization}
\label{sec:genmulti}

As mentioned in \cref{sec:introduction}, for much of the paper we will
work in the extra generality of ``twisted group actions''.
Specifically, let \bW be a complete and cocomplete closed symmetric
monoidal category, and let $G$ be a group that acts on \bW by strong
symmetric monoidal functors.  We write the action of $g\in G$ on
$W\in \bW$ as $\act g W$.  For simplicity, we suppose that the action
is strict, i.e.\ $\pact h {\act g W} = \act{gh}{W}$ and $\act 1 W = W$
strictly (and symmetric-monoidal-functorially).

\begin{eg}\label{eg:op-as-gw}
  The case we are most interested in, which will yield our theorems
  about duality involutions on 2-categories, is when $\bW=\bCat$ with
  $G$ the 2-element group $\{+,-\}$ with $+$ the identity element (a
  copy of $\dZ/2\dZ$), and $\act{-}{A}=A\op$.
\end{eg}

However, there are other examples as well.  Here are a few, also
mentioned in \cref{sec:introduction}, that yield ``duality
involutions'' with a similar flavor.

\begin{eg}
  Let $\bW=\mathbf{2}\text{-}\mathbf{Cat}$, with $G$ as the 4-element
  group $\{++,-+,+-,--\}$ (a copy of $\dZ/2\dZ\times \dZ/2\dZ$), and
  $\act{-+}{A}=A\op$, $\act{+-}{A}=A\co$, and hence
  $\act{--}{A}=A\coop$.  If we give $\bW$ the Gray monoidal structure
  as in~\cite{gps:tricats}, this example leads to a theory of duality
  involutions on Gray-categories.
\end{eg}

\begin{eg}
  Let $\bW$ be the category of strict $n$-categories, with
  $G = (\dZ/2\dZ)^n$ acting by reversal of $k$-morphisms at all
  levels.  Since a category enriched in strict $n$-categories is
  exactly a strict $(n+1)$-category, we obtain a theory of duality
  involutions on strict $(n+1)$-categories.
\end{eg}

\begin{eg}
  Let $\bW=\mathbf{sSet}$, the category of simplicial sets, with
  $G=\{+,-\}$, and $\act{-}{A}$ obtained by reversing the directions
  of all simplices in $A$.  This leads to a theory of duality
  involutions on simplicially enriched categories that is appropriate
  when the simplicial sets are regarded as modeling
  $(\infty,1)$-categories as
  quasicategories~\cite{joyal:q_kan,lurie:higher-topoi}, so that
  simplicially enriched categories are a model for
  $(\infty,2)$-categories.  For example, such simplicially enriched
  categories are used in~\cite{rv:fibyon-oocosmos} to define a notion
  of ``$\infty$-cosmos'' analogous to the ``fibrational cosmoi''
  of~\cite{street:fib-yoneda-2cat}, so such duality involutions could
  be a first step towards an $\infty$-version
  of~\cite{weber:2toposes}.
\end{eg}

\begin{eg}
  Combining the ideas of the last two examples, if $\bW$ is the
  category of $\Theta_n$-spaces as in~\cite{rezk:theta}, then
  $(\dZ/2\dZ)^n$ acts on it by reversing directions at all dimensions.
  Thus, we obtain a theory of duality involutions on categories
  enriched in $\Theta_n$-spaces, which in~\cite{br:cmp-infn-i} were
  shown to be a model of $(\infty,n+1)$-categories.
\end{eg}

Note that we do \emph{not} assume the action of $G$ on \bW is by
\bW-\emph{enriched} functors, since in most of the above examples this
is not the case.  In particular, $(-)\op$ is not a 2-functor.  We also
note that most or all of the theory would probably be the same if $G$
were a 2-group rather than just a group, but we do not need this extra
generality.

Since the action of $G$ on \bW is symmetric monoidal, it extends to an
action on \wcat applied homwise, which we also write $\act{g}{\cA}$,
i.e.\ $\act{g}{\cA}(x,y) = \pact{g}{\cA(x,y)}$.  In our motivating
example \ref{eg:op-as-gw} we have $\act{-}{\cA}= \cA\co$ for a
2-category \cA.  We now define a 2-monad $T$ on $\wcat$ by
\[ T\cA = \sum_{g\in G} \act{g}{\cA}. \]
The unit $\cA \to T\cA$ includes the summand indexed by $1\in G$, and
the multiplication uses the fact that each action, being an
equivalence of categories (indeed, an isomorphism of categories), is
cocontinuous:
\begin{equation}
  TT\cA = \sum_{g\in G} \pact{g}{T\cA}
  = \sum_{g\in G} \act{g}{\left(\sum_{h\in G} \act{h}{\cA}\right)}
  \cong \sum_{g\in G} \sum_{h\in G} \pact{g}{\act{h}{\cA}}
  \cong \sum_{g\in G} \sum_{h\in G} \act{hg}{\cA}
\end{equation}
which we can map into $T\cA$ by sending the $(g,h)$ summand to the $hg$-summand.

We will refer to a normal pseudo $T$-algebra structure as a
\textbf{twisted $G$-action}; it equips a \bW-category \cA with actions
$\pact{g}{-} : \act{g}{\cA}\to\cA$ that are suitably associative up to
coherent isomorphism (with $\pact{1}{x}=x$ strictly).  In our
motivating example of $\bW=\bCat$ and $G=\{+,-\}$, the monad $T$
agrees with the one we constructed in \cref{sec:2-monadic-approach};
thus twisted $G$-actions are strong duality involutions (and likewise
for their morphisms and 2-cells).

\begin{eg}
  If we write $[x,y]$ for the internal-hom of \bW, then we have maps
  $\act{g}{[x,y]} \to [\act g x, \act g y]$ obtained by adjunction
  from the composite
  \[ \act{g}{[x,y]} \otimes \act g x \toiso \pact g{[x,y]\otimes x}
  \to \act g y\]
  Since the $[x,y]$ are the hom-objects of the \bW-category \bW, these
  actions assemble into a \bW-functor
  $\pact{g}{-}:\act g \bW \to \bW$, and as $g$ varies they give \bW
  itself a twisted $G$-action.  (Thus, among the three different
  actions we are denoting by $\pact{g}{-}$ --- the given one on \bW,
  the induced one on \wcat, and an arbitrary twisted $G$-action ---
  the first is a special case of the third.)  In particular, we obtain
  in this way the canonical strong (in fact, strict) duality
  involution on \bCat.
\end{eg}

Now we note that $T$ extends to a normal monad in the sense
of~\cite{cs:multicats} on the proarrow equipment \wprof, as follows.
As in~\cite{shulman:frbi,cs:multicats}, we view equipments as pseudo
double categories satisfying with a ``fibrancy'' condition saying that
horizontal arrows (the ``proarrow'' direction, for us) can be pulled
back universally along vertical ones (the ``functor'' direction).  In
\wprof the objects are \bW-categories, a horizontal arrow
$\cA\hto \cB$ is a profunctor (i.e.\ a \bW-functor
$\cB\op\ten\cA \to \bW$), a vertical arrow $\cA\to\cB$ is a
\bW-functor, and a square
\begin{equation}
  \vcenter{\xymatrix{
      \cA\ar[r]|{|}^M\ar[d]_F \ar@{}[dr]|{\Downarrow} &
      \cB\ar[d]^G\\
      \cC\ar[r]|{|}_N &
      \cD
      }}
\end{equation}
is a \bW-natural transformation $M(b,a)\to N(G(b),F(a))$.  A monad on
an equipment is strictly functorial in the vertical direction, laxly
functorial in the horizontal direction, and its multiplication and
unit transformations consist of vertical arrows and squares.

In our case, we already have the action of $T$ on \bW-categories and
\bW-functors.  A \bW-profunctor $M:\cA \hto \cB$ induces another one
$\act{g}{M} : \act g \cA \hto \act g \cB$ by applying the $G$-action
objectwise, and by summing up over $g$ we have
$T M : T \cA \hto T \cB$.  This is in fact pseudofunctorial on
profunctors.  Finally, the unit and multiplication are already defined
as vertical arrows, and extend to squares in an evident way:
\begin{equation}
  \vcenter{\xymatrix{
      \cA\ar[r]|{|}^M\ar[d]_\eta \ar@{}[dr]|{\Downarrow} &
      \cB\ar[d]^\eta\\
      T\cA\ar[r]|{|}_{T M} &
      T\cB
    }}\qquad
  \vcenter{\xymatrix{
      TT\cA\ar[r]|{|}^{TTM}\ar[d]_\mu \ar@{}[dr]|{\Downarrow} &
      TT\cB\ar[d]^\mu\\
      T\cA\ar[r]|{|}_{T M} &
      T\cB
      }}
\end{equation}

Since we have a monad on an equipment, we can define
``$T$-multicategories'' in \wprof, which following~\cite{cs:multicats}
we call \emph{virtual $T$-algebras}.  For our specific monad $T$, we
will refer to virtual $T$-algebras as \textbf{$G$-variant
  \bW-categories}.  Such a gadget is a \bW-category \cA together with
a profunctor $\A:\cA \hto T\cA$, a unit isomorphism
$\cA(x,y) \toiso \A(\eta(x),y)$, and a composition
\begin{equation}
  \vcenter{\xymatrix{
      \cA \ar[r]|{|}^\A\ar@{=}[d] \ar@{}[drr]|{\Downarrow} &
      T\cA \ar[r]|{|}^{T\A} &
      TT\cA\ar[d]^\mu\\
      \cA\ar[rr]|{|}_\A & &
      T\cA
      }}
\end{equation}
satisfying associativity and unit axioms.  If we unravel this
explicitly, we see that a $G$-variant \bW-category has a set of
objects along with, for each pair of objects $x,y$ and each $g\in G$,
a hom-object $\A^g(x,y)\in \bW$, plus units $\unit \to \A^1(x,x)$ and
compositions
\[\A^g(y,z) \otimes \pact{g}{\A^h(x,y)} \to \A^{hg}(x,z)\]
satisfying the expected axioms.  (Technically, in addition to the
hom-objects $\A^1(x,y)$ it has the hom-objects $\cA(x,y)$ that are
isomorphic to them, but we may ignore this duplication of data.)  We
may refer to the elements of $\A^g(x,y)$ as \textbf{$g$-variant
  morphisms}.  The rule for the variance of composites is easier to
remember when written in diagrammatic order: if we denote
$\al\in \A^g(x,y)$ by $\al:x\tovar{g}y$, then the composite of
$x\tovar{g} y \tovar{h} z$ is $x\tovar{gh} z$.  (Of course, in our
motivating example $G$ is commutative, so the order makes no
difference.)

In the specific case of $G=\{+,-\}$ acting on $\bCat$, we can unravel
the definition more explicitly as follows.
\begin{defn}\label{defn:2cat-contra}
  A \textbf{2-category with contravariance} is a $G$-variant
  \bW-category for $\bW=\bCat$ and $G=\{+,-\}$.  Thus it consists of
  \begin{itemize}[leftmargin=2em]
  \item A collection $\ob\A$ of objects;
  \item For each $x,y\in\ob\A$, a pair of categories $\A\p(x,y)$ and $\A\m(x,y)$;
  \item For each $x\in\ob\A$, an object $1_x \in \A\p(x,x)$;
  \item For each $x,y,z\in\ob\A$, composition functors
    \begin{align*}
      \A\p(y,z) \times \A\p(x,y) &\too \A\p(x,z)\\
      \A\m(y,z) \times \A\m(x,y)\op &\too \A\p(x,z)\\
      \A\p(y,z) \times \A\m(x,y) &\too \A\m(x,z)\\
      \A\m(y,z) \times \A\p(x,y)\op &\too \A\m(x,z);
    \end{align*}
  \end{itemize}
  such that
  \begin{itemize}[leftmargin=2em]
  \item Four ($2\cdot 2^1$) unitality diagrams commute; and
  \item Eight ($2^3$) associativity diagrams commute.
  \end{itemize}
\end{defn}
Like any kind of generalized multicategory, $G$-variant \bW-categories
form a 2-category.  We leave it to the reader to write out explicitly
what the morphisms and 2-cells in this 2-category look like; in our
example of interest we will call them \textbf{2-functors preserving
  contravariance} and \textbf{2-natural transformations respecting
  contravariance}.

Now, according to~\cite[Theorem 9.2]{cs:multicats}, any twisted
$G$-action $a:T\cA\to \cA$ gives rise to a $G$-variant \bW-category
with $\A = \cA(a,1)$, which in our situation means
$\A^g(x,y) = \cA(\act g x,y)$ (where $\act g x$ refers, as before, to
the $g$-component of the action $a:T\cA\to \cA$).  In particular, any
2-category with a strong duality involution can be regarded as a
2-category with contravariance, where $\A\p(x,y) = \cA(x,y)$ and
$\A\m(x,y) = \cA(x\o,y)$.

Moreover, by~\cite[Corollary 9.4]{cs:multicats}, a $G$-variant
\bW-category \A arises from a twisted $G$-action exactly when
\begin{enumerate}
\item The profunctor $\A : \cA \hto T\cA$ is representable by some $a:T\cA\to\cA$, and\label{item:rep1}
\item The induced 2-cell $\overline{a} : a\circ \mu \to a \circ Ta$ is an isomorphism.\label{item:rep2}
\end{enumerate}
Condition~\ref{item:rep1} means that for every $x\in \A$ and every
$g\in G$, there is an object ``$\act g x$'' and an isomorphism
$\A^g(x,y) \cong \A^1(\act g x,y)$, natural in $y$.  The Yoneda lemma
implies this isomorphism is mediated by a ``universal $g$-variant
morphism'' $\chi_{g,x} \in \A^g(x,\act g x)$.

Condition~\ref{item:rep2} then means that for any $x\in \A$ and
$g,h\in G$, the induced map
$\psi_{h,g,x}:\act{gh}{x} \to \pact{h}{\act g x}$ is an isomorphism.
(This map arises by composing $\chi_{g,x} \in \A^g(x,\act g x)$ with
$\chi_{h,\act g x} \in \A^h(\act g x,\pact{h}{\act g x})$ to obtain a
map in $\A^{gh}(x,\pact{h}{\act g x})$, then applying the defining
isomorphism of $\act{gh}{x}$.)  As usual for generalized
multicategories, this is equivalent to requiring a stronger universal
property of $\act g x$: that precomposing with $\chi_{g,x}$ induces
isomorphisms
\begin{equation}
  \A^h(\act g x,y) \toiso \A^{gh}(x,y)\label{eq:gvariator}
\end{equation}
for all $h\in G$.  (This again is more mnemonic in diagrammatic
notation: any arrow $x \tovar{gh} y$ factors uniquely through
$\chi_{g,x}$ by a morphism $\act g x \tovar{h} y$, i.e.\ variances on
the arrow can be moved into the action on the domain, preserving
order.)  This is because the following diagram commutes by definition
of $\psi_{h,g,x}$, and the vertical maps are isomorphisms by
definition of $\chi$:
\begin{equation}
  \vcenter{\xymatrix@C=4pc{
      \A^{gh}(x,y)\ar@{<-}[r]^{-\circ \chi_{g,x}}\ar@{<-}[d]_{-\circ \chi_{gh,x}} &
      \A^h(\act g x,y)\ar@{<-}[d]^{-\circ \chi_{h,\act g x}}\\
      \A^1(\act{gh}x,y)\ar@{<-}[r]_{-\circ \psi_{h,g,x}} &
      \A^1(\pact h{\act g x},y)
      }}
\end{equation}
If $\act g x$ is an object equipped with a morphism
$\chi_{g,x} \in \A^g(x,\act g x)$ satisfying this stronger universal
property~\eqref{eq:gvariator}, we will call it a \textbf{$g$-variator}
of $x$.  In our motivating example $\bW=\bCat$ with $g=-$, we call a
$-$-variator an \textbf{opposite}.  Explicitly, this means the
following.

\begin{defn}\label{defn:opposite}
  In a 2-category with contravariance \A, a \textbf{(strict) opposite}
  of an object $x$ is an object $x\o$ equipped with a contravariant
  morphism $\chi_x\in \A^-(x,x\o)$ such that precomposing with
  $\chi_x$ induces isomorphisms of hom-categories for all $y$:
  \begin{align*}
    \A^+(x\o,y) &\toiso \A^-(x,y)\\
    \A^-(x\o,y) &\toiso \A^+(x,y).
  \end{align*}
\end{defn}

In fact, $g$-variators can also be characterized more explicitly.  The
second universal property of $\chi_{g,x}\in \act g \A(x,\act g x)$
means in particular that the identity $1_x\in \act 1 \A(x,x)$ can be
written as $\xi_{g,x} \circ \chi_{g,x}$ for a unique
$\xi_{g,x} \in \A^{g^{-1}}(\act g x,x)$.  (This is the first place
where we have used the fact that $G$ is a group rather than just a
monoid.)  Moreover, since
\[(\chi_{g,x}\circ \xi_{g,x}) \circ \chi_{g,x} = \chi_{g,x} \circ
(\xi_{g,x}\circ \chi_{g,x}) = \chi_{g,x} \]
it follows by the first universal property of $\chi_{g,x}$ that
$\chi_{g,x}\circ \xi_{g,x} = 1_{\act g x}$ as well.  Thus,
$\chi_{g,x}$ and $\xi_{g,x}$ form a ``$g$-variant isomorphism''
between $x$ and $\act g x$.

On the other hand, it is easy to check that any such $g$-variant
isomorphism between $x$ and an object $y$ makes $y$ into a
$g$-variator of $x$.  Thus, we have:

\begin{prop}\label{thm:gm-abs}
  Any $g$-variant \bW-functor $F:\A\to\B$ preserves $g$-variators.
  In particulary, any 2-functor preserving contravariance also preserves opposites.
\end{prop}
\begin{proof}
  It obviously preserves ``$g$-variant isomorphisms''.
\end{proof}

Thus we have:

\begin{thm}\label{thm:gm-2cat}
  The 2-category of 2-categories with strong duality involutions,
  duality 2-functors, and duality 2-natural transformations is
  2-equivalent to the 2-category of 2-categories with contravariance
  in which every object has a strict opposite, 2-functors preserving
  contravariance, and 2-natural transformations respecting
  contravariance.
\end{thm}
\begin{proof}
  By~\cite[Theorem 9.13]{cs:multicats} and the remarks preceding
  \cref{defn:opposite}, the latter 2-category is equivalent to the
  2-category of pseudo $T$-algebras, \emph{lax} $T$-morphisms, and
  $T$-2-cells.  However, \cref{thm:gm-abs} implies that in fact every
  lax $T$-morphism is a pseudo $T$-morphism.  Finally, every pseudo
  $T$-algebra is isomorphic to a normal pseudo one obtained by
  re-choosing $\act{1}{(-)}$ to be the identity (which it is assumed
  to be isomorphic to).
\end{proof}

\section{Contravariance through enrichment}
\label{sec:contr-thro-enrichm}

We continue with our setup from \cref{sec:genmulti}, with a complete
and cocomplete closed monoidal category \bW and a group $G$ acting on
\bW.  We start by noticing that the monad $T$ on \wprof constructed in
\cref{sec:genmulti} can actually be obtained in a standard way from a
simpler monad.

Recall that there is another equipment \wmat whose objects are sets,
whose vertical arrows are functions, and whose horizontal arrows
$X\hto Y$ are ``\bW-valued matrices'', which are just functions
$Y\times X \to \bW$; we call them matrices because we compose them by
``matrix multiplication''.  The equipment \wprof is obtained from
\wmat by applying a functor \dMod that constructs monoids (monads) and
modules in the horizontal directions
(see~\cite{shulman:frbi,cs:multicats}).  We now observe that our monad
$T$, like many monads on equipments of profunctors, is also in the
image of \dMod.

Let $S$ be the following monad on \wmat.
On objects and vertical arrows, it acts by $S(X)=X\times G$.
On a \bW-matrix $M:Y\times X \to \bW$ it acts by
\[SM((y,h),(x,g)) =
  \begin{cases}
    \pact g{M(y,x)} &\quad g=h\\
    \emptyset &\quad g\neq h
  \end{cases}
\]
We may write this schematically using a Kronecker delta as
\[SM((y,h),(x,g)) = \delta_{g,h}\cdot \pact g{M(y,x)}.\]
On a composite of matrices $X \xhto{M} Y \xhto{N} Z$ we have
\begin{align*}
  (SM \odot SN)((z,k),(x,g)) &= \sum_{(y,h)} (\delta_{g,h}\cdot \act g{M(y,x)}) \otimes (\delta_{h,k}\cdot \act h{N(z,y)})\\
  &\cong \delta_{k,g}\sum_{y} \act g{M(y,x)} \otimes \act g{N(z,y)}\\
  &\cong \delta_{k,g} \Bigpact g{\sum_{y} \big(M(y,x) \otimes N(z,y)\big)}\\
  &= \delta_{k,g}\cdot \act g{(M\odot N)(z,x)}\\
  &= S(M\odot N)((z,k),(x,g))
\end{align*}
making $S$ a pseudofunctor.  The monad multiplication and unit are
induced from the multiplication and unit of $G$; the squares
\begin{equation}
  \vcenter{\xymatrix{
      X\ar[r]|{|}^M\ar[d]_\eta \ar@{}[dr]|{\Downarrow} &
      Y\ar[d]^\eta\\
      SX\ar[r]|{|}_{S M} &
      SY
    }}\qquad
  \vcenter{\xymatrix{
      SSX\ar[r]|{|}^{SSM}\ar[d]_\mu \ar@{}[dr]|{\Downarrow} &
      SSY\ar[d]^\mu\\
      SX\ar[r]|{|}_{S M} &
      SY
      }}
\end{equation}
map the components $M(y,x)$ and $\pact{h}{\act{g}{M(y,x)}}$
isomorphically to $\act 1{M(y,x)}$ and $\act{gh}{M(y,x)}$
respectively.

Now, recalling that $T\cA = \sum_{g\in G} \act g \cA$, we see that
$\ob(T\cA) = \ob(\cA) \times G$ and
\[T\cA((y,h),(x,g)) = \delta_{h,g}\cdot \pact g{\cA(y,x)},\]
and so in fact $T \cong \dMod(S)$.  Thus, by~\cite[Theorem
8.7]{cs:multicats}, virtual $T$-algebras can be identified with
``$S$-monoids''; these are defined like virtual $S$-algebras, with
sets and matrices of course replacing categories and profunctors, and
omitting the requirement that the unit be an isomorphism.  Thus, an
$S$-monoid consists of a set $X$ of objects, a function
$\A:S(X)\times X = X\times G\times X \to\bW$, unit maps
$1_x:I\to \A^1(x,x)$, and composition maps that turn out to look like
$\A^g(y,z) \otimes \pact g{\A^h(x,y)} \to \A^{hg}(x,z)$.  Note that
this is exactly what we obtain from a virtual $T$-algebra by omitting
the redundant data of the hom-objects $\cA(x,y)$ and their
isomorphisms to $\A^1(x,y)$; this is essentially the content
of~\cite[Theorem 8.7]{cs:multicats} in our case.

In~\cite{cs:multicats}, the construction of $S$-monoids is factored
into two: first we build a new equipment $\hkl(\wmat,S)$ whose objects
and vertical arrows are the same as \wmat but whose horizontal arrows
$X\hto Y$ are the horizontal arrows $X\hto SY$ in \wmat, and then we
consider horizontal monoids in $\hkl(\wmat,S)$.  In fact,
$\hkl(\wmat,S)$ is in general only a ``virtual equipment'' (i.e.\ we
cannot compose its horizontal arrows, though we can ``map out of
composites'' like in a multicategory), but in our case it is an
ordinary equipment because $S$ is ``horizontally
strong''~\cite[Theorem A.8]{cs:multicats}.  This means that $S$ is a
strong functor (which we have already observed) and that the induced
maps of matrices
\begin{align*}
  (\eta,1)_!M &\to (1,\eta)^*S M \\
  (\mu,1)_!SSM &\to (1,\mu)^*\odot SSM
\end{align*}
are isomorphisms, where $f^*$ and $f_!$ denote the pullback and its
left adjoint pushforward of matrices along functions.  Indeed, we have
\begin{align*}
  (\eta,1)_!M((y,h),x) &= \delta_{h,1} \cdot M(y,x)\hspace{2cm}\text{while}\\
  (1,\eta)^*S M((y,h),x) &= SM((y,h),(x,1))\\
  &= \delta_{h,1} \cdot \pact 1{M(y,x)}\\
  &= \delta_{h,1} \cdot M(y,x)
\end{align*}
and likewise
\begin{align*}
  (\mu,1)_!SSM((y,h),((x,g_1),g_2)) &= \textstyle\sum_{h_2 h_1 = h} SSM(((y,h_1),h_2),((x,g_1),g_2))\\
  &= \textstyle\sum_{h_1 h_2 = h} \delta_{h_2,g_2} \cdot \Bigpact{g_2}{\delta_{h_1,g_1} \cdot \act{g_1}{M(y,x)}}\\
  &= \textstyle\sum_{h_1 h_2 = h} \delta_{h_2,g_2}\delta_{h_1,g_1} \cdot \act{g_1 g_2}{M(y,x)}\\
  &= \delta_{h,g_1 g_2} \cdot \act{g_1 g_2}{M(y,x)}\\
  \intertext{while}
  (1,\mu)^*\odot SSM((y,h),((x,g_1),g_2)) &= SSM((y,h),(x,g_1 g_2))\\
  &= \delta_{h,g_1 g_2} \cdot \act{g_1 g_2}{M(y,x)}.
\end{align*}
Inspecting the definition of composition in~\cite[Appendix
A]{cs:multicats}, we see that the composite of $M:X\hto SY$ and
$N:Y\hto SZ$ is
\[ (M\odot_S N)((z,h),x) =  \sum_y \sum_{g_1g_2 = h} M((y,g_1),x) \odot \act{g_1}{N((z,g_2),y)}
\]
Note that what comes after the $\sum_y$ depends only on $M((y,-),x)$
and $N((z,-),y)$, which are objects of $\bW^G$.  Thus, if we write
\twgw for the category $\bW^G$ with the following monoidal structure:
\[ (M \otimes N)(h) = \sum_{g_1 g_2 = h} M(g_1) \odot \act{g_1}{N(g_2)}
\]
then we have $\hkl(\wmat,S) \cong \twgwmat$.  It follows that
$S$-monoids (that is, $G$-variant \bW-categories) can equivalently be
regarded as ordinary monoids in the equipment \twgwmat.  But since
monoids in an equipment of matrices are simply enriched categories, we
can identify $G$-variant \bW-categories with $\twgw$-enriched
categories.

Note that this monoidal structure on \twgw is \emph{not} symmetric.
It is a version of Day convolution~\cite{day:refl-closed} that is
``twisted'' by the action of $G$ on \bW (see~\cite{loregian:twgw} for
further discussion).  Like an ordinary Day convolution monoidal
structure, it is also closed on both sides (as long as \bW is); that
is, we have left and right hom-functors $\homlbare$ and $\homrbare$
with natural isomorphisms
\begin{equation}\label{eq:biclosed}
  (\twgw)(A\otimes B, C) \cong (\twgw)(A, \homr B C) \cong (\twgw)(B, \homl A C).
\end{equation}
Inspecting the definition of the tensor product in \twgw, it suffices to define
\begin{align*}
  (\homr B C)(g) &\coloneqq \prod_{h} \big(\homr{\act g{B(h)}}{C(gh)}\big)\\
  (\homl A C)(g) &\coloneqq \prod_{h} \left(\homl{\act {h^{-1}}{A(h)}}{\act{h^{-1}}{C(hg)}}\right)
\end{align*}
(This is another place where we use the fact that $G$ is a group
rather than a monoid.)  As usual, it follows that \twgw can be
regarded as a \twgw-category (that is, as a $g$-variant \bW-category),
with hom-objects $\underline{\twgw}(A,B) \coloneqq (\homr A B)$.  (The
fact that a closed monoidal category becomes self-enriched is often
described only for closed \emph{symmetric} monoidal categories, but it
works just as well for closed non-symmetric ones, as long as we use
the \emph{right} hom.)

Bringing things back down to each a bit, in our specific case with
$\bW=\bCat$ and $G=\{+,-\}$, let us write
$\V = \int_{\{+,-\}}\!\bCat$.  The underlying category of \V is just
$\bCat\times\bCat$, but we denote its objects as $A = (A\p,A\m)$, with
$A\p$ the \emph{covariant part} and $A\m$ the \emph{contravariant
  part}.
The monoidal structure on \V is the following nonstandard one:
\begin{align*}
(A\otimes B)\p &\coloneqq \big(A\p \times B\p\big) \amalg \big(A\m \times (B\m)\op\big)\\
(A\otimes B)\m &\coloneqq \big(A\p \times B\m\big) \amalg \big(A\m \times (B\p)\op\big)
\end{align*}
The unit object is
\[ \unit \coloneqq (1,0) \]
where $1$ denotes the terminal category and $0$ the initial (empty) one.
The conclusion of our equipment-theoretic digression above is then the following:

\begin{thm}\label{thm:contrav-enriched}
  The 2-category of 2-categories with contravariance, 2-functors
  preserving contravariance, and 2-natural transformations respecting
  contravariance is 2-equivalent to the 2-category of \V-enriched
  categories.\qed
\end{thm}

This theorem is easy to prove explicitly as well, of course.  A
\V-category has, for each pair of objects $x,y$, a pair of
hom-categories $(\A\p(x,y),\A\m(x,y))$, together with composition
functors that end up looking just like those in
\cref{defn:2cat-contra}, and so on.  But I hope that the digression
makes this theorem seem less accidental; it also makes it clear how to
generalize it to other examples.

The underlying ordinary category $\A_{\,o}$ of a 2-category \A with
contravariance, in the usual sense of enriched category theory,
consists of its objects and its covariant 1-morphisms (the objects of
the categories $\A\p(x,y)$).  It also has an underlying ordinary
2-category, induced by the lax monoidal forgetful functor
$(-)\p : \V \to \bCat$, whose hom-categories are the categories
$\A\p(x,y)$; we denote this 2-category by $\A\p$.  Of course, there is
no 2-category to denote by ``$\A\m$'', but we could say for instance
that $\A\m$ is a profunctor from $\A\p$ to itself.

\section{Opposites through enrichment}
\label{sec:opposites}

For most of this section, we let $(\V,\otimes,\unit)$ be an arbitrary
biclosed monoidal category, not assumed symmetric.  We are, of course,
thinking of our \V from the last section, or more generally \twgw.

Suppose \A is a \V-category, that $x\in\ob\A$, and $\om\in\ob\V$.  A
\textbf{copower} (or \textbf{tensor}) of $x$ by $\om$ is an object
$\om\odot x$ of $\A$ together with isomorphisms in \V:
\begin{equation}\label{eq:copower}
  \A(\om\odot x,y) \cong \homr \om {\A(x,y)}
\end{equation}
for all $y\in\ob\A$, which are \emph{\V-natural} in the sense that for
any $y,z\in\ob\A$, the following diagram commutes:
\[\vcenter{\xymatrix{
    \A(y,z) \otimes \A(\om\odot x,y)\ar[r]^-\cong \ar[d] &
    \A(y,z) \otimes (\homr \om {\A(x,y)}) \ar[r] &
    \homr \om {(\A(y,z)\otimes \A(x,y))} \ar[d]\\
    \A(\om\odot x, z)\ar[rr]_\cong &&
    \homr \om{\A(x,z)}
  }}\]
Taking $y=\om\odot x$ in~\eqref{eq:copower}, we obtain from
$1_{\om\odot x}$ a canonical map $\om \to \A(x,\om\odot x)$, which by
the Yoneda lemma determines~\eqref{eq:copower} uniquely.  Of course,
this is just the usual definition of copowers in enriched categories,
specialized to enrichment over \V.  We have spelled it out explicitly
to emphasize that it makes perfect sense even though \V is not
symmetric, as long as we choose the correct hom $\homrbare$ and not
$\homlbare$ (see~\cite{street:absolute}, which treats the even more
general case of enrichment over a \emph{bicategory}).

Note that if $\A=\uV$ (the category \V regarded as a \V-category),
then the tensor product $\om\otimes x$ is a copower $\om\odot x$.
Moreover, for general \A, if $\om,\varpi\in\V$ and the copowers
$\om\odot x$ and $\varpi\odot (\om\odot x)$ exist, we have
\begin{align}
  \A(\varpi \odot (\om\odot x), y)
  &\cong \homr \varpi {\A(\om\odot x,y)}\\
  &\cong \homr \varpi (\homr \om {\A(x,y)})\\
  &\cong \homr {(\varpi \otimes \om)} {\A(x,y)}
\end{align}
so that $\varpi\odot (\om\odot x)$ is a copower
$(\varpi\otimes \om) \odot x$.  In particular, these observations
mandate writing the copower as $\om\odot x$ rather than $x\odot \om$.

Frequently one defines a \emph{power} in a \V-category \A to be a
copower in $\A\op$, but since our \V is not symmetric, \V-categories
do not have opposites.  Thus, we must define directly a \textbf{power}
of $x$ by $\om$ to be an object $\pow \om x\in \ob\A$ together with
isomorphisms
\begin{equation}
  \A(y,\pow \om x) \cong \homl \om {\A(y,x)}
\end{equation}
for all $y\in\ob\A$, which are \V-natural in that the following diagram commutes:
\[\vcenter{\xymatrix{
    \A(y,\pow\om x) \otimes \A(z,y)\ar[r]^-\cong\ar[d] &
    (\homl \om {\A(y,x)}) \otimes \A(z,y)\ar[r] &
    \homl \om {(\A(y,x) \otimes \A(z,y))}\ar[d]\\
    \A(z,\pow\om x) \ar[rr]_-\cong & &
    \homl \om {\A(z,x)}
    }}
\]
Analogous arguments to those for copowers show that when $\A=\uV$,
then $\homl\om x$ is a power $\pow \om x$, and that in general we have
$\pow\varpi {(\pow \om x)}\cong \pow{(\om\otimes\varpi)} x$.  If both
the copower $\om\odot x$ and the power $\pow\om x$ exist, then we have
\begin{align}
  \A_{\,o}(\om\odot x,y)
  &\cong \V(\unit,{\A(\om\odot x,y)})\\
  &\cong \V(\unit, \homr\om{\A(x,y)})\\
  &\cong \V(\om,{\A(x,y)})\\
  &\cong \V(\unit,\homl\om{\A(x,y)})\\
  &\cong \V(\unit,\A(x,\pow\om y))\\
  &\cong \A_{\,o}(x,\pow\om y).
\end{align}
so that the endofunctors $(\om\odot -)$ and $(\pow\om-)$ on the
underlying 1-category $\A_{\,o}$ are adjoint.  They are \emph{not}
adjoint \V-functors, even when $\A=\uV$: in our motivating example,
the isomorphisms~\eqref{eq:biclosed} do not even lift from the
1-category $\uV_{\,o}=\V$ to the 2-category $\uV\p$.

Now suppose that $\om$ is \emph{right dualizable} in \V, i.e.\ that we
have an object $\om^*\in\V$ and morphisms $\om^* \otimes \om\to \unit$
and $\unit\to\om\otimes\om^*$ satisfying the triangle identities.
Then $(-\otimes\om^*)$ is right adjoint to $(-\otimes\om)$, hence
isomorphic to $(\homr\om-)$; and dually we have
$(\om\otimes-) \cong (\homl{\om^*}-)$.  Thus, a copower $\om\odot x$
in a \V-category \A is equivalently characterized by an isomorphism
\begin{equation}
  \A(\om\odot x,-) \cong \A(x,-) \otimes \om^*,\label{eq:abscopower}
\end{equation}
while a power $\pow{\om^*}x$ is characterized by an isomorphism
\begin{equation}
  \A(-,\pow{\om^*} x) \cong \om\otimes \A(-,x).\label{eq:abspower}
\end{equation}
However, for fixed $x$, the right-hand sides of~\eqref{eq:abscopower}
and~\eqref{eq:abspower} are adjoint in the bicategory of \V-modules.
Since $\A(\om\odot x,-)$ always has an adjoint $\A(-,\om\odot x)$, and
likewise $\A(-,\pow{\om^*} x)$ always has an adjoint
$\A(\pow{\om^*} x,-)$, it follows that giving a copower $\om\odot x$
is equivalent to giving a power $\pow{\om^*}x$.

Now let us specialize to the case of $\twgw$.  Then for any $g\in G$,
we have a \textbf{twisted unit} $\twu{g} \in\twgw$, defined by
$\twu{g}(h) = \delta_{g,h} \cdot \unit$.  By definition of $\homrbare$
and $\homlbare$, we have
\begin{align}
  (\homr {\twu{h}} {\A(x,y)})(g) \;&\cong\; \A^{gh}(x,y) 
  \qquad\text{and}\\
  (\homl {\twu{h}} {\A(x,y)})(g) \;&\cong\; \act{h^{-1}}{(\A^{hg}(x,y))}. 
\end{align}
Thus, $\twu{h}\odot x$, if it exists, is characterized by isomorphisms
\begin{equation}
  \A^g(\twu{h}\odot x,y) \cong \A^{gh}(x,y)
\end{equation}
that are suitably and jointly natural in $y$.  In other words, a
copower $\twu{h}\odot x$ is precisely an \emph{$h$-variator} of $x$ as
defined in \cref{sec:genmulti}.  And in our particular case of
$\bW=\bCat$, a copower $\du\odot x$ is precisely an \emph{opposite} of
$x$ as defined in \cref{defn:opposite}.  Thus we have:

\begin{thm}
  A 2-category with contravariance has opposites, as in
  \cref{defn:opposite}, exactly if when regarded as a \V-category it
  has all copowers by $\du$.\qed
\end{thm}

Note that since $\twu{h}\otimes\twu{h^{-1}} \cong \unit$, in
particular $\twu{h}$ is dualizable.  Thus, copowers by $\twu{h}$ are
equivalent to powers by $\twu{h^{-1}}$.  In particular, since
$-\in \{+,-\}$ is its own inverse, it follows that $\du$ is self-dual,
and opposites are also characterized by isomorphisms
\begin{equation}
  \A\p(y,x\o) \cong \A\m(y,x)\op \qquad\text{and}\qquad \A\m(y,x\o) \cong \A\p(y,x)\op.
\end{equation}
This gives another reason why a 2-functor preserving contravariance
must preserve opposites: copowers by a dualizable object are absolute
colimits~\cite{street:absolute}.

\section{Bicategories with contravariance}
\label{sec:bicat-contra}

We have now reached the top of the right-hand side of the ladder from
\cref{sec:introduction}.  It remains to move across to the other side
and head down, starting with a bicategorical version of \V-categories
for our $\V=\int_{\{+,-\}}\bCat$.

In fact, it will be convenient to stay in a more general setting.
Thus, suppose that our monoidal category \bW is actually a 2-category
\cW, and that our group $G$ acts on it by 2-functors.  In this case,
the construction of \twgw can all be done with 2-categories, obtaining
a monoidal 2-category \twgcw (in the strict sense of a monoidal
\bCat-enriched category).  Since a monoidal 2-category is \textit{a
  fortiori} a monoidal bicategory, we can consider \twgcw-enriched
bicategories, which we call \textbf{$g$-variant \cW-bicategories}.

The most comprehensive extant reference on enriched bicategories seems
to be~\cite{gs:freecocomp}, though the basic definition dates back at
least to~\cite{carmody:thesis,lack:thesis}.  The definition of an
enriched bicategory is quite simple: we just write out the definition
of bicategory and replace all hom-categories by objects of \twgcw,
cartesian products of categories by $\otimes$, and functors and
natural transformations by morphisms and 2-cells in \twgcw.  If we
write this out explicitly, it consists of the following.
\begin{itemize}[leftmargin=2em]
\item A collection $\ob\A$ of objects;
\item For each $x,y\in\ob\A$ and $g\in G$, a category $\A^g(x,y)$;
\item For each $x\in\ob\A$, a unit morphism $1_x : \unit \to \A^1(x,x)$;
\item For each $x,y,z\in\ob\A$ and $g,h\in G$, composition morphisms
  \[ \A^h(y,z) \otimes \pact{h}{\A^g(x,y)} \to \A^{gh}(x,z) \]
\item For each $x,y\in \ob\A$ and $g\in G$, two natural unitality isomorphisms;
\item For each $x,y,z,w\in\ob\A$ and $g,h,k\in G$, an associativity isomorphism;
\item For each $x,y,z\in \ob\A$ and $g,h\in G$, a unitality axiom holds; and
\item For each $x,y,z,w,u\in\ob\A$ and $g,h,k,\ell\in G$, an associativity pentagon holds.
\end{itemize}
Enriched bicategories, of course, come naturally with a notion of
enriched functor.  (In fact, as described in~\cite{gs:freecocomp} we
have a whole tricategory of enriched bicategories, but we will not
need the higher structure.)  Explicitly, a \twgcw-enriched functor
$F:\A\to\B$ consists of
\begin{itemize}[leftmargin=2em]
\item A function $F:\ob\A\to\ob\B$; and
\item For each $x,y\in\ob\A$ and $g\in G$, a morphism $F:\A^g(x,y)\to\B^g(Fx,Fy)$;
\item For each $x\in\ob\A$, an isomorphism $F(1_x) \cong 1_{F x}$;
\item For each $x,y,z\in \ob\A$ and $g,h\in G$, a natural functoriality isomorphism of the form $(Fg)(Ff) \cong F(gf)$;
\item For each $x,y\in \ob\A$ and $g\in G$, a unit coherence diagram commmutes;
\item For each $x,y,z,w\in \ob\A$ and $g,h,k\in G$, an associativity coherence diagram commutes.
\end{itemize}

In the case of interest, we have $\bW=\bCat$, which is of course
enhances to the 2-category \cCat.  However, we cannot take
$\cW=\cCat$, because as we have remarked, $(-)\op$ is not a 2-functor
on \cCat, so $\{+,-\}$ does not act on \cCat through 2-functors.
However, $(-)\op$ is a 2-functor on $\cCat_g$, the 2-category of
categories, functors, and natural \emph{isomorphisms}; so this is what
we take as our $\cW$.  We denote the resulting monoidal 2-category
\twgcw by \cV, and make the obvious definition:

\begin{defn}
  A \textbf{bicategory with contravariance} is a \cV-enriched
  bicategory, and a \textbf{pseudofunctor preserving contravariance}
  is a \cV-enriched functor.
\end{defn}

If we write this out explicitly in terms of covariant and
contravariant parts, we see that a bicategory with contravariance has
four kinds of composition functors, eight kinds of associativity
isomorphisms, and sixteen coherence pentagons.  Working with an
abstract \cW and $G$ thus allows us to avoid tedious case-analyses.

We now generalize the enriched notion of $g$-variator (and hence of
``opposite'') from \cref{sec:opposites} to the bicategorical case.
For any $\om\in\twgcw$, any \twgcw-bicategory \A, and any $x\in \A$, a
\textbf{copower} of $x$ by $\om$ is an object $\om\odot x$ together
with a map $\om \to \A(x,\om\odot x)$ such that for any $y$ the
induced map $\A(\om\odot x,y) \to \homr \om {\A(x,y)}$ is an
\emph{equivalence} (not necessarily an isomorphism).  (This is
essentially the special case of~\cite[10.1]{gs:freecocomp} when \cB is
the unit \twgcw-bicategory.)

As in \cref{sec:opposites}, we are mainly interested in the case when
$\om$ is one of the {twisted units} $\twu{g}$.  In this case we again
write $\act g x$ for $\twu{g}\odot x$, and the map
$\twu{g} \to \A(x,\act g x)$ is just a $g$-variant morphism
$\chi_{g,x}\in\A^g(x,\act g x)$.  Its universal property says that any
$gh$-variant morphism $x\tovar{gh} y$ factors essentially uniquely
through $\chi_{g,x}$ via an $h$-variant morphism
$\act g x \tovar{h} y$ (and similarly for 2-cells); that is, we have
equivalences
\[ \A^h(\act g x,y) \toiso \A^{gh}(x,y)
\]
As before, by Yoneda arguments this is equivalent to having a
$g$-variant morphism $x \tovar{g} \act g x$ and a $g^{-1}$-variant
morphism $\act g x \tovar{g^{-1}} x$ whose composites in both
directions are isomorphic to identities; that is, a ``$g$-variant
equivalence''.  In the specific example of $\cW=\cCat_g$ and
$G=\{+,-\}$, we of course call $\act - x$ a \textbf{(weak) opposite}
of $x$, written $x\o$.

Our goal now is to show that any weak duality involution on a
bicategory \cA gives it the structure of a bicategory with
contravariance having weak opposites; but to minimize case analyses,
we will work in the generality of \cW and $G$.  Thus, we first define
a \textbf{(weak, strictly unital) twisted $G$-action} on a
\cW-category \cA to consist of:
\begin{itemize}[leftmargin=2em]
\item For each $g\in G$, a \cW-functor
  $\pact g{-}:\act g \cA \to \cA$.  (Note that here $\act g \cA$
  denotes the hom-wise action, $\act g \cA(x,y) = \act{g}{\cA(x,y)}$.)
  When $g=1$ is the unit element of $G$, we ask that $\pact 1{-}$ be
  exactly equal to the identity functor.
\item For each $g,h\in G$, a \cW-pseudonatural adjoint equivalence
  \[ \xymatrix{
      \pact{gh}{\cA} \ar[dr]_{\pact{h}{\pact{g}{-}}} \ar[rr]^{\pact{gh}{-}} && \cA.\\
      & \act h \cA \ar[ur]_{\pact{h}{-}} \ar@{}[u]|(.6){\Downarrow\fy}
    }
  \]
  (Note that ${\pact{h}{\pact{g}{-}}}$ means the homwise endofunctor
  $\pact{h}{-}$ of $\cW$-bicategories applied to the action functor
  $\pact g{-}:\act g \cA \to \cA$.)  When $g$ or $h$ is $1\in G$, we
  ask that $\fy$ be exactly the identity transformation.
\item For each $g,h,k\in G$, an invertible \cW-modification
  \[ \vcenter{\xymatrix{
      &\act{hk}{\cA} \ar[dr] \ar[dd] \\
      \act{ghk}{\cA} \ar[ur] \ar[dr] \ar@{}[r]|(.6){\act k {\fy_{g,h}}} &
      \ar@{}[r]|(.4){\fy_{h,k}}& \cA \\
      & \act{k}{\cA} \ar[ur] }}
    \quad\overset{\zeta}{\Longrightarrow}\quad
    \vcenter{\xymatrix{&\act{hk}{\cA} \ar[dr] \ar@{}[d]|(.6){\fy_{g,hk}} \\
        \act{ghk}{\cA} \ar[ur] \ar[dr] \ar[rr] &
        \ar@{}[d]|(.4){\fy_{gh,k}} & \cA \\ 
        & \act{k}{\cA} \ar[ur] }}
  \]
  As before, when $g$, $h$, or $k$ is $1\in G$, we ask that $\ze$ be
  exactly the identity.
\item For each $g,h,k,\ell\in G$, a 4-simplex diagram of instances of
  $\zeta$ commutes.
\end{itemize}
In our motivating example of $\cW=\cCat_g$ and $G=\{+,-\}$, the strict
identity requirements mean that:
\begin{itemize}[leftmargin=2em]
\item The only nontrivial action is $\pact{-}{-}$, which we write as
$(-)\o$.
\item The only nontrivial $\fy$ is $\fy_{-,-}$, which has the same
type as the $\fy$ in \cref{defn:duality-involution}.
\item The only nontrivial $\ze$ is $\ze_{-,-,-}$, which has an
equivalent type to the $\ze$ in \cref{defn:duality-involution} (since
$--=+$ is the identity, $\fy_{-,--}$ and $\fy_{--,-}$ are identities,
so the type of $\ze$ displayed above has moved one copy of $\fy$ from
the codomain to the domain).
\item The only nontrivial axiom likewise reduces to the one given in
\cref{defn:duality-involution}.
\end{itemize}
Thus, this really does generalize our notion of duality involution.
Now we will show:

\begin{thm}\label{thm:bicat}
  Let \cA be a \cW-bicategory with a twisted $G$-action, and for
  $x,y\in \cA$ and $g\in G$ define $\A^g(x,y) = \cA(\act g x,y)$.
  Then $\A$ is a \twgcw-bicategory with copowers by all the twisted
  units.
\end{thm}
\begin{proof}
  We define the composition morphisms as follows:
  \[\begin{array}{rcl}
      \A^h(y,z) \otimes \pact{h}{\A^g(x,y)}
      &=& \cA(\act h y,z) \otimes \pact{h}{\cA(\act g x,y)}\\
      &=& \cA(\act h y,z) \otimes \act{h}\cA(\act g x,y)\\
      &\xto{\pact h{-}}& \cA(\act h y,z) \otimes \cA(\pact{h}{\act g x},\act h y)\\
      &\xto{\mathrm{comp}}& \cA(\pact{h}{\act g x}, z)\\
      &\xto{-\circ \fy_{g,h}}& \cA(\act{gh}x,z)\\
      &=&\A^{gh}(x,z)
    \end{array}\]
    Informally (or, formally, in an appropriate internal ``linear type
theory'' of \cW), we can say that the composite of $\beta\in\A^h(y,z)$
and $\alpha\in \pact h{\A^g(x,y)}$ is
  \[\beta \circ \act h\alpha \circ \fy_{g,h}\]
  where $\circ$ denotes composition in \cA.  Expressed in the same
  way, the associator for $\al\in \pact{hk}{\A^g(x,y)}$,
  $\be\in \pact{k}{\A^h(y,z)}$, and $\gm\in\A^k(z,w)$ is
  \begin{align*}
    (\gm \circ \act k \beta \circ \fy_{h,k}) \circ \act {hk} \alpha \circ \fy_{g,hk}
    &\cong \gm\circ \act k \beta \circ \pact k {\act h \alpha} \circ \fy_{h,k} \circ \fy_{g,hk}\\
    &\cong \gm\circ \act k \beta \circ \pact k {\act h \alpha} \circ \act k {\fy_{g,h}} \circ \fy_{gh,k}\\
    &\cong \gm \circ \pact k{\beta \circ \act h \alpha \circ \fy_{g,h}} \circ \fy_{gh,k}
  \end{align*}
  using the naturality of \fy, the modification $\ze$, and the
  functoriality of $\pact k-$ (and omitting the associativity
  isomorphisms of \cA, by coherence for bicategories).

  For the unit, since $\A^1(x,y)= \cA(x,y)$, the unit map
  $\unit \to \A^1(x,y)$ is just the unit of \cA.  One unit isomorphism
  is just that of \cA, while the other is that of \cA together with
  the unit isomorphism of the pseudofunctor $\pact g{-}$.  And the
  associator appearing in the unit axiom has $g=k=1$, so all the
  $\fy$'s collapse and it is essentially trivial, and the unit axiom
  follows immediately from that of \cA.

  To show that \A is a \twgcw-bicategory, it remains to consider the
  pentagon axiom.  Omitting $\circ$ from now on, the pentagon axiom is
  an equality of two morphisms
  \[ \delta \act \ell \gamma \fy_{k,\ell} \act{k\ell}\beta \fy_{h,k\ell} \act{hk\ell}\alpha \fy_{g,hk\ell}
    \longrightarrow
    \delta\pact \ell{\gamma\pact k{\beta \act h \alpha \fy_{g,h}} \fy_{gh,k}} \fy_{ghk,\ell}
    \]
    By naturality of the functoriality isomorphisms for the actions
    $\pact{g}{-}$, we can certainly push all applications of them to
    the end where they will be equal; thus it suffices to compare the
    morphisms
  \[ \delta \act \ell \gamma \fy_{k,\ell} \act{k\ell}\beta \fy_{h,k\ell} \act{hk\ell}\alpha \fy_{g,hk\ell}
    \longrightarrow
    \delta\act \ell\gamma \pact \ell{\act k{\beta}} \pact\ell{\pact k{\act h \alpha}} \pact \ell{\act k{\fy_{g,h}}} \act\ell{\fy_{gh,k}} \fy_{ghk,\ell}
    \]
    This is done in Figure~\ref{fig:pentagon}, where most of the
    regions are naturality, except for the one at the bottom left
    which is the 4-simplex axiom for $\ze$.
  \begin{figure}
    \centering
    \[\hspace{-2cm}\xymatrix@C=1pc{
        \delta \act \ell \gamma \fy_{k,\ell} \act{k\ell}\beta \fy_{h,k\ell} \act{hk\ell}\alpha \fy_{g,hk\ell}
        \ar[d] \ar[r] &
        \delta \act \ell \gamma  \pact\ell{\act{k}\beta} \fy_{k,\ell} \fy_{h,k\ell} \act{hk\ell}\alpha \fy_{g,hk\ell}
        \ar[r]^-\ze \ar[d] &
        \delta \act \ell \gamma  \pact\ell{\act{k}\beta} \act\ell{\fy_{h,k}} \fy_{hk,\ell} \act{hk\ell}\alpha \fy_{g,hk\ell}
        \ar[dd]\\
        \delta \act \ell \gamma \fy_{k,\ell} \act{k\ell}\beta  \pact{k\ell}{\act{h}\alpha} \fy_{h,k\ell} \fy_{g,hk\ell}
        \ar[d]_\ze \ar[r] &
        \delta \act \ell \gamma \pact\ell{\act{k}\beta} \fy_{k,\ell}   \pact{k\ell}{\act{h}\alpha} \fy_{h,k\ell} \fy_{g,hk\ell}
        \ar[ddl]^\ze \ar[dd] &
        \\
        \delta \act \ell \gamma \fy_{k,\ell} \act{k\ell}\beta  \pact{k\ell}{\act{h}\alpha} \act{k\ell}{\fy_{g,h}} \fy_{gh,k\ell}
        \ar[d] 
        & &
        \delta \act \ell \gamma  \pact\ell{\act{k}\beta} \act\ell{\fy_{h,k}} \pact\ell{\act{hk}\alpha} \fy_{hk,\ell} \fy_{g,hk\ell}
        \ar[dd]^\ze \ar[ddl]
        \\
        \delta \act \ell \gamma \pact\ell{\act{k}\beta} \fy_{k,\ell}\pact{k\ell}{\act{h}\alpha} \act{k\ell}{\fy_{g,h}} \fy_{gh,k\ell}
        \ar[d] &
        \delta \act \ell \gamma \pact\ell{\act{k}\beta} \pact{\ell}{\pact{k}{\act{h}\alpha}} \fy_{k,\ell} \fy_{h,k\ell} \fy_{g,hk\ell}
        \ar[d]^\ze \ar[dl]^\ze
        & \\
        \delta \act \ell \gamma \pact\ell{\act{k}\beta} \pact{\ell}{\pact{k}{\act{h}\alpha}} \fy_{k,\ell} \act{k\ell}{\fy_{g,h}} \fy_{gh,k\ell}
        \ar[d] &
        \delta \act \ell \gamma  \pact\ell{\act{k}\beta} \pact\ell{\pact{k}{\act{h}\alpha}} \act\ell{\fy_{h,k}} \fy_{hk,\ell} \fy_{g,hk\ell}
        \ar[d]_{\ze}
        &
        \delta \act \ell \gamma  \pact\ell{\act{k}\beta} \act\ell{\fy_{h,k}} \pact\ell{\act{hk}\alpha} \pact\ell{\fy_{g,hk}} \fy_{ghk,\ell}
        \ar[dl] \\
        \delta \act \ell \gamma \pact\ell{\act{k}\beta} \pact{\ell}{\pact{k}{\act{h}\alpha}} \pact\ell{\act{k}{\fy_{g,h}}} \fy_{k,\ell} \fy_{gh,k\ell}
        \ar[d]_{\ze} &
        \delta \act \ell \gamma  \pact\ell{\act{k}\beta} \pact\ell{\pact k {\act{h}\alpha}} \act\ell{\fy_{h,k}} \pact\ell{\fy_{g,hk}} \fy_{ghk,\ell}
        \ar[dl]_-{\act\ell\ze}
        &
        \\
        \delta\act \ell\gamma \pact \ell{\act k{\beta}} \pact\ell{\pact k{\act h \alpha}} \pact \ell{\act k{\fy_{g,h}}} \act\ell{\fy_{gh,k}} \fy_{ghk,\ell}
        & &
      }\hspace{-2cm} \]
    \caption{The pentagon axiom}
    \label{fig:pentagon}
  \end{figure}

  Now we must show that \A has copowers by the twisted units; of
  course we will use $\act g x$ as the copower $\twu{g} \odot x$.
  Since $\A^g(x,\act g x) = \cA(\act g x,\act g x)$ by definition, for
  $\chi_{g,x}$ we can take the identity map of $\act g x$ in \cA.  By
  definition of composition in \A, the induced precomposition map
  \[\A^h(\act g x,y) \to \A^{gh}(x,y)\]
  is essentially just precomposition with $\fy$:
  \[ \cA(\pact h{\act g x},y) \to \cA(\act{gh} x,y) \]
  and hence is an equivalence.
  Thus, \A has copowers by the twisted units.
\end{proof}

Inspecting the construction, we also conclude:

\begin{sch}\label{thm:bicat-2cat}
  If a 2-category \cA has a twisted $G$-action in the sense of
  \cref{sec:genmulti}, and we regard this as a weak twisted $G$-action
  in the sense defined above with the actions being strict functors,
  $\fy$ strictly natural, and $\ze$ an identity, then the
  \twgcw-bicategory constructed in \cref{thm:bicat} is actually a
  strict \twgw-category, and this construction agrees with the one in
  \S\ref{sec:genmulti}--\ref{sec:contr-thro-enrichm}.  In particular,
  if \cA is a 2-category with a strong duality involution, and we
  regard it as a bicategory with a weak duality involution to
  construct a bicategory with contravariance \A, the result is the
  2-category with contravariance we already obtained from it in
  \cref{sec:genmulti}.\qed
\end{sch}

With some more work we could enhance \cref{thm:bicat} to a whole
equivalence of tricategories.  However, all we will need for our
coherence theorem, in addition to \cref{thm:bicat} and
\cref{thm:bicat-2cat}, is to go backwards on biequivalences.

Before stating such a theorem, we have to define what we want to get
out of it.  Suppose \cA and \cB are \cW-bicategories with twisted
$G$-action; by a \textbf{twisted $G$-functor} $F:\cA\to\cB$ we mean a
functor of \cW-bicategories together with:
\begin{itemize}[leftmargin=2em]
\item For each $g\in G$, a \cW-pseudonatural adjoint equivalence
  \[\vcenter{\xymatrix{
        \act g\cA\ar[d]_{\pact g -}\ar[r]^-{\act g F} \drtwocell\omit{\fii} &
        \act g \cB\ar[d]^{\pact g -}\\
        \cA\ar[r]_-F &
        \cB.
      }}\]
\item For each $g,h\in G$, an invertible \cW-modification
  \[\vcenter{\xymatrix@R=3pc@C=3pc{
        && {\act {gh}\cA} \ar[dl]_{\pact h{\pact g -}} \ar[r]^{\act {gh}F} \dtwocell\omit{\fii}
        & \act {gh}\cB \ar[dl]|{\pact h{\pact g -}} \ddluppertwocell^{\mathrlap{\pact{gh}{-}}}{\fy}\\
        &\act h\cA\ar[d]_{\pact h -}\ar[r]|-{\act h F} \drtwocell\omit{\fii} &
        \act h\cB\ar[d]|{\pact h -}\\
        &\cA\ar[r]_-F &
        \cB
      }}
    \quad \overset{\theta}{\Longrightarrow}\quad
    \vcenter{\xymatrix@R=3pc@C=3pc{
        & {\act {gh}\cA} \ar[dl]|{\pact h{\pact g -}} \ddluppertwocell^{\mathrlap{\pact{gh}{-}}}{\fy}
        \ar[r]^{\act {gh}F} &
        \act{gh}{\cB}  \ddluppertwocell^{\mathrlap{\pact{gh}{-}}}{\fii} \\
        \act h\cA\ar[d]|{\pact h -}\\
        \cA \ar[r]_{F} & \cB
      }}
  \]
  (As before, $\pact h{\pact g -}$ denotes the functorial action of
  the homwise endofunctor $\pact h -$ of \cW-bicategories on the given
  action functor $\pact g -: \act h \cA \to \cA$.)  This can be
  written formally as
  \[ \fii_h \circ \act h {\fii_g} \circ \fy^\cB_{g,h} \cong \fy^\cA_{g,h} \circ \fii_{gh} \]
\item For all $g,h,k\in G$, an axiom holds that can be written
  formally as the commutative diagram shown in Figure~\ref{fig:theta-ax}.
  \begin{figure}
    \centering
    \[ \xymatrix{\fii_k \act k {\fii_h} \pact k{\act h {\fii_g}} \fy^\cB_{h,k} \fy^\cB_{g,hk} \ar[r] \ar[d]_\ze &
      \fii_k \act k {\fii_h} \fy^\cB_{h,k} \act {hk}{\fii_g} \fy^\cB_{g,hk} \ar[r]^\theta &
      \fy^\cA_{h,k} \fii_{hk} \act {hk}{\fii_g} \fy^\cB_{g,hk} \ar[d]^\theta \\
      \fii_k \act k {\fii_h} \pact k{\act h {\fii_g}} \act k{\fy^\cB_{g,h}} \fy^\cB_{gh,k} \ar[d] &
      & \fy^\cA_{h,k} \fy^\cA_{g,hk} \fii_{ghk} \ar[d]^\zeta \\
      \fii_k \pact k {\fii_h {\act h {\fii_g}} {\fy^\cB_{g,h}}} \fy^\cB_{gh,k} \ar[d]_\theta  & &
      \pact k {\fy^\cA_{g,h}} \fy^\cA_{gh,k} \fii_{ghk}\\
      \fii_k \pact k { {\fy^\cA_{g,h}} \fii_{gh}} \fy^\cB_{gh,k} \ar[r] &
      \fii_k \pact k {\fy^\cA_{g,h}} \pact k{\fii_{gh}} \fy^\cB_{gh,k} \ar[r] &
      \pact k {\fy^\cA_{g,h}} \fii_k \pact k{\fii_{gh}} \fy^\cB_{gh,k} \ar[u]_\theta
      & }
    \]
    \caption{The axiom for $\theta$}
    \label{fig:theta-ax}
  \end{figure}
\end{itemize}

\begin{thm}\label{thm:biequiv}
  Suppose \cA and \cB are \cW-bicategories with twisted $G$-action,
  with resulting \twgcw-bicategories \A and \B.  If \A and \B are
  biequivalent as \twgcw-bicategories, then \cA and \cB are
  biequivalent by a twisted $G$-functor.
  
  In particular, if two bicategories \cA and \cB with duality
  involution give rise to biequivalent
  bicategories-with-contravariance, then \cA and \cB are biequivalent
  by a duality pseudofunctor.
\end{thm}
\begin{proof}
  Let $F:\A\to\B$ be a \twgcw-biequivalence.  In particular, then, it
  is a biequivalence on the $1$-parts, hence a biequivalence
  $\cA\simeq \cB$.
  
  Now by \cref{thm:bicat}, for any $x\in \cA$ we have a ``$g$-variant
  equivalence'' $x\tovar{g} \act g x$ with inverse
  $\act g x \tovar{g^{-1}} x$.  This structure is preserved by $F$, so
  we have a $g$-variant equivalence between $Fx$ and $F(\act g x)$.
  But we also have a $g$-variant equivalence between $Fx$ and
  $\pact{g}{Fx}$, and composing them we obtain an ordinary
  ($1$-variant) isomorphism $\pact{g}{Fx}\cong F(\act g x)$.  These
  supply the components of $\fii$; their pseudo\-naturality is
  straightforward.

  Now, by construction of the copowers by twisted units, it follows
  that $\fy_{g,h} : \act{gh}{x} \to \pact h {\act g x}$ is isomorphic
  to the composite of the variant equivalences
  \[ \act{gh}{x} \tovar{(gh)^{-1}} x \tovar{g} \act g x \tovar{h} \pact h{\act g x}
  \]
  while $\ze$ is obtained by canceling and uncanceling some of these
  equivalences.  In particular, when $\fy$ is composed with $\fii$, we
  can simply cancel some inverse variant equivalences to obtain the
  components of $\theta$.  As for Figure~\ref{fig:theta-ax}, its
  source is
  \begin{multline*}
    \pact{ghk}{F(x)}
    \tovar{(ghk)^{-1}} F(x)
    \tovar{g} \pact g{F(x)}
    \tovar{hk} \pact{hk}{\pact g{F(x)}}\\
    \tovar{(hk)^{-1}} \pact g {F(x)}
    \tovar{h} \pact h{\pact g {F(x)}}
    \tovar{k} \pact k{\pact h{\pact g {F(x)}}}\\
    \tovar{g^{-1}} \pact k{\pact h{F(x)}}
    \tovar{g} \pact k{\pact h{F({\act g x})}}\\
    \tovar{h^{-1}} \pact k{F({\act g x})}
    \tovar{h} \pact k{F(\pact h{\act g x})}
    \tovar{k^{-1}} {F(\pact h{\act g x})}
    \tovar{k} F(\pact k{\pact h{\act g x}})
  \end{multline*}
  while its target is
  \begin{multline*}
    \pact{ghk}{F(x)}
    \tovar{(ghk)^{-1}} F(x)
    \tovar{ghk} F(\act{ghk} x)\\
    \tovar{(ghk)^{-1}} F(x)
    \tovar{gh} F(\act{gh}x)
    \tovar{k} F(\pact{k}{\act{gh}{x}})\\
    \tovar{(gh)^{-1}} F(\act{k}{x})
    \tovar{g} F(\pact{k}{\act{g}{x}})
    \tovar{h} F(\pact{k}{\pact{h}{\act{g}x}})
  \end{multline*}
  Here we have applied functors such as $\pact k-$ to variant
  morphisms; we can define this by simply ``conjugating'' with the
  variant equivalences $x\tovar{k} \act k x$.  We leave it to the
  reader to apply naturality and cancel all the redundancy in these
  composites, reducing them both to
  \[ \pact{ghk}{F(x)} \tovar{(ghk)^{-1}} F(x)
    \tovar{g} F({\act{g}{x}})
    \tovar{h} F(\pact{h}{\act{g}x})
    \tovar{k} F(\pact{k}{\pact{h}{\act{g}x}}) \]
  so that they are equal.
\end{proof}

Therefore, to strictify a bicategory with duality involution, it will
suffice to strictify its corresponding bicategory with contravariance.
This is the task of the next, and final, section.

\section{Coherence for enriched bicategories}
\label{sec:coherence}

We could continue in the generality of $G$ and \cW, but there seems
little to be gained by it any more.

\begin{thm}\label{thm:bicat-coherence}
  Any bicategory with contravariance is biequivalent to a 2-category
  with contravariance.
\end{thm}
\begin{proof}
  Just as there are two ways to prove the coherence theorem for
  ordinary bicategories, there are two ways to prove this coherence
  theorem.  The first is an algebraic one, involving formally adding
  strings of composable arrows that hence compose strictly
  associatively.  This can be expressed abstractly using the same
  general coherence theorem for pseudo-algebras over a 2-monad that we
  used in \cref{sec:2-monadic-approach}.  As sketched at the end
  of~\cite[\S4]{shulman:psalg}, this theorem (or a slight
  generalization of it) applies as soon as we observe that our
  2-category \cV is closed monoidal and cocomplete.

  The other method is by a Yoneda embedding.  To generalize this to
  the enriched (and non-symmetric) case, first note that for any
  \cV-bicategory \A, by~\cite[9.3--9.6]{gs:freecocomp} we have a
  \cV-bicategory $\cM\A$ of \emph{moderate \A-modules}, and a Yoneda
  embedding $\A \to \cM\A$ that is fully faithful.  Thus, \A is
  biequivalent to its image in $\cM\A$.  However, since \cV is a
  strict 2-category that is closed and complete, $\cM\A$ is actually a
  strict \cV-category, and hence so is any full subcategory of it.

  Explicitly, an \A-module consists of categories $F\p(x)$ and
  $F\m(x)$ for each $x\in\A$ together with actions
  \begin{align*}
    F\p(y) \times \A\p(x,y) &\to F\p(x)\\
    F\p(y) \times \A\m(x,y) &\to F\m(x)\\
    F\m(y) \times \A\p(x,y)\op &\to F\m(x)\\
    F\m(y) \times \A\m(x,y)\op &\to F\p(x)
  \end{align*}
  and coherent associativity and unitality isomorphisms.  A covariant
  \A-module morphism consists of functors $F\p(x) \to G\p(x)$ and
  $F\m(x)\to G\m(x)$ that commute up to coherent natural isomorphism
  with the actions, while a contravariant one consists similarly of
  functors $F\p(x)\op \to G\m(x)$ and $F\m(x)\op\to G\p(x)$.  Since
  \cCat is a strict 2-category, the bicategory-with-contravariance of
  modules is in fact a strict 2-category with contravariance.  The
  Yoneda embedding, of course, sends each $z\in \A$ to the
  representable $Y_z$ defined by $Y_z\p(x) \coloneqq \A\p(x,z)$ and
  $Y_z\m(x)\coloneqq \A\m(x,z)$.
\end{proof}

We have almost completed our trip over the ladder; it remains to make
the following observation and then put all the pieces together.

\begin{thm}\label{thm:opposites-weaktostrict}
  If \A is a 2-category with contravariance that has \emph{weak}
  opposites, then it is biequivalent to a 2-category with
  contravariance having strict opposites.
\end{thm}
\begin{proof}
  Let $\A'$ be the free cocompletion of \A, as a strict \cV-category,
  under strict opposites (a strict \cV-weighted colimit).  It is easy
  to see that this can be done in one step, by considering the
  collection of all opposites of representables in the presheaf
  \cV-category of \A.  Thus, the embedding $\A\to\A'$ is
  \cV-fully-faithful, and every object of $\A'$ is the strict opposite
  of something in the image.  However, \A has weak opposites, which
  are preserved by any \cV-functor, and any strict opposite is a weak
  opposite.  Thus, every object of $\A'$ is equivalent to something in
  the image of \A, since they are both a weak opposite of the same
  object.  Hence $\A\to\A'$ is bicategorically essentially surjective,
  and thus a biequivalence.
\end{proof}

Finally, we can prove \cref{thm:main}.

\begin{thm}\label{thm:main2}
  If $\cA$ is a bicategory with a weak duality involution, then there
  is a 2-category $\cA'$ with a strict duality involution and a
  duality pseudofunctor $\cA\to\cA'$ that is a biequivalence.
\end{thm}
\begin{proof}
  By \cref{thm:bicat}, we can regard \cA as a bicategory with
  contravariance \A having weak opposites.  By
  \cref{thm:bicat-coherence}, it is therefore biequivalent to a
  2-category with contravariance and weak opposites, and therefore by
  \cref{thm:opposites-weaktostrict} also biequivalent to a 2-category
  with contravariance and strict opposites.

  Now by \cref{thm:gm-2cat} and \cref{thm:contrav-enriched}, the
  latter is equivalently a 2-category with a strong duality
  involution.  Thus, by \cref{thm:2monad-coherence} it is equivalent
  to a 2-category with a strict duality involution, say $\cA'$.  So we
  have a biequivalence $\cA\to\cA'$ that is a pseudofunctor preserving
  contravariance, and by \cref{thm:biequiv}, we can also regard it as
  a duality pseudofunctor.
\end{proof}

As mentioned in \cref{sec:introduction}, we could actually dispense
with the right-hand side of the ladder as follows.  Let $\cA'$ be the
full sub-\V-bicategory of $\cM\cA$, as in \cref{thm:bicat-coherence},
consisting of the modules that are \emph{either} of the form $Y_z$ or
of the form $Y_z\o$, where $Y_z\o$ is defined by
$(Y_z\o)\p(x) \coloneqq \A\m(x,z)\op$ and
$(Y_z\o)\m(x) \coloneqq \A\p(x,z)\op$.  This $\cA'$ is a 2-category
with a strict duality involution that interchanges $Y_z$ and $Y_z\o$,
and the Yoneda embedding is a biequivalence for the same reasons.
However, this quicker argument still depends on the description of
weak duality involutions using bicategorical enrichment from
\cref{sec:bicat-contra}, and thus still depends \emph{conceptually} on
the entire picture.

\begin{references*}

\bibitem[BKP89]{bkp:2dmonads}
R.~Blackwell, G.~M. Kelly, and A.~J. Power.
\newblock Two-dimensional monad theory.
\newblock {\em J. Pure Appl. Algebra}, 59(1):1--41, 1989.

\bibitem[BKPS89]{bkps:flexible}
G.~J. Bird, G.~M. Kelly, A.~J. Power, and R.~H. Street.
\newblock Flexible limits for {$2$}-categories.
\newblock {\em J. Pure Appl. Algebra}, 61(1):1--27, 1989.

\bibitem[BR13]{br:cmp-infn-i}
Julia~E. Bergner and Charles Rezk.
\newblock Comparison of models for $(\infty,n)$-categories, {I}.
\newblock {\em Geometry \& Topology}, 17:2163--2202, 2013.
\newblock arXiv:1204.2013.

\bibitem[Car95]{carmody:thesis}
S.M. Carmody.
\newblock {\em Cobordism Categories}.
\newblock PhD thesis, University of Cambridge, 1995.

\bibitem[CS10]{cs:multicats}
G.S.H. Cruttwell and Michael Shulman.
\newblock A unified framework for generalized multicategories.
\newblock {\em Theory Appl. Categ.}, 24:580--655, 2010.
\newblock arXiv:0907.2460.

\bibitem[Day72]{day:refl-closed}
Brian Day.
\newblock A reflection theorem for closed categories.
\newblock {\em J. Pure Appl. Algebra}, 2(1):1--11, 1972.

\bibitem[DS97]{ds:monbi-hopfagbd}
Brian Day and Ross Street.
\newblock Monoidal bicategories and {H}opf algebroids.
\newblock {\em Adv. Math.}, 129(1):99--157, 1997.

\bibitem[Gar14]{garner:F}
Richard Garner.
\newblock Lawvere theories, finitary monads and {Cauchy}-completion.
\newblock {\em Journal of Pure and Applied Algebra}, 218(11):1973--1988, 2014.

\bibitem[Gar18]{garner:emb-tancat}
Richard Garner.
\newblock An embedding theorem for tangent categories.
\newblock {\em Advances in Mathematics}, 323:668 -- 687, 2018.

\bibitem[GP97]{gp:enr-var}
R.~Gordon and A.~J. Power.
\newblock Enrichment through variation.
\newblock {\em J. Pure Appl. Algebra}, 120(2):167--185, 1997.

\bibitem[GP17]{gp:enr-lawvere}
Richard Garner and John Power.
\newblock An enriched view on the extended finitary monad--{Lawvere} theory
  correspondence.
\newblock arXiv:1707.08694, 2017.

\bibitem[GPS95]{gps:tricats}
R.~Gordon, A.~J. Power, and Ross Street.
\newblock Coherence for tricategories.
\newblock {\em Mem. Amer. Math. Soc.}, 117(558):vi+81, 1995.

\bibitem[GS16]{gs:freecocomp}
Richard Garner and Michael Shulman.
\newblock Enriched categories as a free cocompletion.
\newblock {\em Adv. Math}, 289:1--94, 2016.
\newblock arXiv:1301.3191.

\bibitem[Gur12]{gurski:biequiv}
Nick Gurski.
\newblock Biequivalences in tricategories.
\newblock {\em Theory and Applications of Categories}, 26(14):329--384, 2012.

\bibitem[Her01]{hermida:coh-univ}
Claudio Hermida.
\newblock From coherent structures to universal properties.
\newblock {\em J. Pure Appl. Algebra}, 165(1):7--61, 2001.

\bibitem[Joy02]{joyal:q_kan}
A.~Joyal.
\newblock Quasi-categories and {K}an complexes.
\newblock {\em Journal of Pure and Applied Algebra}, 175:207--222, 2002.

\bibitem[Lac95]{lack:thesis}
Stephen Lack.
\newblock {\em The algebra of distributive and extensive categories}.
\newblock PhD thesis, University of Cambridge, 1995.

\bibitem[Lac00]{lack:psmonads}
Stephen Lack.
\newblock A coherent approach to pseudomonads.
\newblock {\em Adv. Math.}, 152(2):179--202, 2000.

\bibitem[Lac02]{lack:codescent-coh}
Stephen Lack.
\newblock Codescent objects and coherence.
\newblock {\em J. Pure Appl. Algebra}, 175(1-3):223--241, 2002.
\newblock Special volume celebrating the 70th birthday of Professor Max Kelly.

\bibitem[Lor16]{loregian:twgw}
Fosco Loregian.
\newblock Answer to {MathOverflow} question ``{Twisted} {Day} convolution''.
\newblock
  \url{http://mathoverflow.net/questions/233812/twisted-day-convolution}, March
  2016.

\bibitem[LS12]{ls:limlax}
Stephen Lack and Michael Shulman.
\newblock Enhanced 2-categories and limits for lax morphisms.
\newblock {\em Advances in Mathematics}, 229(1):294--356, 2012.
\newblock arXiv:1104.2111.

\bibitem[Lur09]{lurie:higher-topoi}
Jacob Lurie.
\newblock {\em Higher topos theory}.
\newblock Number 170 in Annals of Mathematics Studies. Princeton University
  Press, 2009.

\bibitem[Mak96]{makkai:avoiding-choice}
M.~Makkai.
\newblock Avoiding the axiom of choice in general category theory.
\newblock {\em J. Pure Appl. Algebra}, 108(2):109--173, 1996.

\bibitem[Pow89]{power:coherence}
A.~J. Power.
\newblock A general coherence result.
\newblock {\em J. Pure Appl. Algebra}, 57(2):165--173, 1989.

\bibitem[Pow07]{power:3dmonads}
John Power.
\newblock Three dimensional monad theory.
\newblock In {\em Categories in algebra, geometry and mathematical physics},
  volume 431 of {\em Contemp. Math.}, pages 405--426. Amer. Math. Soc.,
  Providence, RI, 2007.

\bibitem[Pro96]{pronk:bicat-frac}
Dorette~A. Pronk.
\newblock Etendues and stacks as bicategories of fractions.
\newblock {\em Compositio Math.}, 102(3):243--303, 1996.

\bibitem[Rez10]{rezk:theta}
Charles Rezk.
\newblock A cartesian presentation of weak $n$-categories.
\newblock {\em Geometry \& Topology}, 14, 2010.
\newblock arXiv:0901.3602.

\bibitem[Rob12]{roberts:ana}
David~M. Roberts.
\newblock Internal categories, anafunctors and localisations.
\newblock {\em Theory and Applications of Categories}, 26(29):788--829, 2012.
\newblock arXiv:1101.2363.

\bibitem[RV17]{rv:fibyon-oocosmos}
Emily Riehl and Dominic Verity.
\newblock Fibrations and {Yoneda's} lemma in an $\infty$-cosmos.
\newblock {\em J. Pure Appl. Algebra}, 221(3):499--564, 2017.
\newblock arXiv:1506.05500.

\bibitem[Shu08]{shulman:frbi}
Michael Shulman.
\newblock Framed bicategories and monoidal fibrations.
\newblock {\em Theory and Applications of Categories}, 20(18):650--738
  (electronic), 2008.
\newblock arXiv:0706.1286.

\bibitem[Shu12]{shulman:psalg}
Michael Shulman.
\newblock Not every pseudoalgebra is equivalent to a strict one.
\newblock {\em Adv. Math.}, 229(3):2024--2041, 2012.
\newblock arXiv:1005.1520.

\bibitem[Sta16]{stay:ccb}
Michael Stay.
\newblock Compact closed bicategories.
\newblock {\em Theory and Applications of Categories}, 31(26):755--798, 2016.
\newblock arXiv:1301.1053.

\bibitem[Str72]{street:ftm}
Ross Street.
\newblock The formal theory of monads.
\newblock {\em J. Pure Appl. Algebra}, 2(2):149--168, 1972.

\bibitem[Str74]{street:fib-yoneda-2cat}
Ross Street.
\newblock Fibrations and {Y}oneda's lemma in a {$2$}-category.
\newblock In {\em Category Seminar (Proc. Sem., Sydney, 1972/1973)}, pages
  104--133. Lecture Notes in Math., Vol. 420. Springer, Berlin, 1974.

\bibitem[Str83]{street:absolute}
Ross Street.
\newblock Absolute colimits in enriched categories.
\newblock {\em Cahiers Topologie G\'eom. Diff\'erentielle}, 24(4):377--379,
  1983.

\bibitem[Web07]{weber:2toposes}
Mark Weber.
\newblock Yoneda structures from 2-toposes.
\newblock {\em Appl. Categ. Structures}, 15(3):259--323, 2007.

\end{references*}

\end{document}